\documentclass[12pt]{article}
\usepackage{amsmath,amsthm,amsfonts,amssymb,color}
\usepackage[T1,T2A]{fontenc}
\usepackage[utf8]{inputenc}
\usepackage[english]{babel}
\usepackage[utf8]{inputenc}
\usepackage{tikz-cd}
\usepackage{mathrsfs}
\usepackage{authblk}

\newcommand{\mC}{\mathbb{C}}

\newcommand{\mR}{\mathbb{R}}

\newcommand{\mZ}{\mathbb{Z}}
\newcommand{\mN}{\mathbb{N}}

\newcommand{\calA}{{\cal A}}
\newcommand{\calB}{{\cal B}}

\newcommand{\calD}{{\cal D}}

\newcommand{\calF}{{\cal F}}
\newcommand{\calH}{{\cal H}}

\newcommand{\calN}{{\cal N}}

\newcommand{\calP}{{\cal P}}
\newcommand{\calQ}{{\cal Q}}

\newcommand{\calS}{{\cal S}}

\newcommand{\calU}{{\cal U}}

\newcommand{\calZ}{{\cal Z}}

\newcommand{\bL}{{\bf L}}
\newcommand{\bF}{{\bf F}}

\newcommand{\bU}{{\bf U}}

\newcommand{\bm}{{\bf m}}

\newcommand{\bu}{{\bf u}}

\newcommand{\bw}{{\bf w}}

\newcommand{\bz}{{\bf z}}
\newcommand{\bkk}{{\bf k}}
\newcommand{\bzz}{{\bf{z}}}
\newcommand{\bxx}{{\bf{x}}}
\newcommand{\bdd}{{\bf{d}}}
\newcommand{\bss}{{\bf{s}}}

\newcommand{\diamd}{{\diamond}}
\newcommand{\pa}{{\partial}}

\theoremstyle{plain}
\newtheorem{theorem}{Theorem}[section]

\newtheorem{lemma}[theorem]{Lemma}

\newtheorem{definition}[theorem]{Definition}

\makeatletter \@addtoreset{equation}{section} \makeatother

\begin{document}

\title{Normalization flow and Poincaré-Dulac theory}
\author{A.\,O.~Chernyshev\\Lomonosov Moscow State University}



\maketitle

\begin{abstract}
In this article, we develop a new approach to the Poincaré--Dulac normal form theory for a system of differential equations near a singular point. Using the continuous averaging method, we construct a normalization flow that moves a vector field to its normal form. We prove that, in the algebra of formal vector fields (given by power series), the normalization procedure achieves full normalization. When convergence is taken into account, we show that the radius of convergence admits a lower bound of order $1/(1+A\delta)$, with $A>0$, as $\delta \to +\infty$. Based on the methods of this work and on the approaches of \cite{Tres2}, we provide a new proof of the Siegel--Brjuno theorem on the convergence of the normalizing transformation.

\medskip
\noindent Bibliography: 16 items.
\end{abstract}

\footnotetext[0]{The author's research was supported by the Theoretical Physics and
Mathematics Advancement Foundation "BASIS".}

\section{Introduction and Basic Definitions}
Consider the system of differential equations
\begin{equation}\label{ODEinitial}
	\dot{z} = \widehat{u}(z),\qquad z\in \mC^n,\qquad \widehat{u}\colon (\mC^n,0) \to (\mC^n,0).
\end{equation}

Assume that the linear part of the vector field $\widehat{u}$ has the following form:
\begin{equation}\label{eigenvalues}
	\frac{\pa \widehat{u}}{\pa z} (0)=\Lambda=\text{diag}(\lambda),\qquad\lambda = (\lambda_1,\ldots,\lambda_n) \in \mC^n.
\end{equation}

In this paper, we address how to construct a continuous normalization process for the system \eqref{ODEinitial} near a singular point. For the system of differential equations \eqref{ODEinitial}, a step-by-step normalization process was proposed by Poincaré (see \cite{Arnold, VanLiChoy}). The Hamiltonian version of normalization was proposed by Birkhoff (see \cite{Birkhoff}). Let $e_s$ be the $s$-th standard basis vector, $s=1,\ldots,n$. We use the notation:
\begin{gather*}
	\bkk = (k_1,\ldots,k_n) \in \mZ^n, \qquad |\bkk| = |k_1|+\ldots+|k_n|,\qquad 	\bzz^{\bkk} = z_1^{k_1}\cdots z_n ^{k_n},  \\
	\bkk_s = \bkk - e_s, \qquad   \langle \lambda,\bkk \rangle = \sum_{j=1}^n \lambda_j k_j.
\end{gather*}

In addition, we set
\begin{equation*}
	\mZ^n _{\diamd} = \{\bkk \in \mZ_{+}^n \colon |\bkk | \geq 2\}.
\end{equation*}

Represent the vector field $\widehat{u}$ in the form
\begin{equation}\label{Lambda}
	\widehat{u}(z) = \Lambda z + \widehat{u}_{\diamd},\quad \widehat{u}_{\diamd}^j =   \sum_{\bkk \in \mZ^n_{\diamd}} \widehat{U}_{\bkk}^j \bzz^{\bkk} e_j,\quad \widehat{U}_{\bkk} = (\widehat{U}_{\bkk}^1,\ldots,\widehat{U}_{\bkk}^n)^T \in \mC^n,
\end{equation}
where $j=1,\ldots,n$, and $e_j$ is the unit vector with $1$ in the $j$-th position.

Define
\begin{equation} \label{eq: defofLlamvda}
	\bL_{\lambda} = \{\bkk \in \mZ_{\diamd}^{n}\colon \langle \lambda, \bkk_s\rangle = 0 \text{ for some } s \in \{1,\ldots, n \} \}.
\end{equation}

We call the monomial $\bzz^{\bkk} e_s$ resonant if $\langle \lambda, \bkk_s \rangle = 0$. Any integer vector $\bkk \in \bL_{\lambda}$ that determines a resonant monomial is called a resonant vector. The number $|\bkk|$ is called the order of the resonance. According to the Poincare-Dulac theory \cite{Arnold, VanLiChoy}, using suitable changes of variables, one can eliminate any finite number of nonresonant monomials of the vector field. To eliminate all nonresonant terms one uses a change of coordinates whose Taylor series generally diverges. The vector field obtained as a result of the normalization procedure is called the normal form of the vector field. In this work we consider changes with the identity linear part.

The convergence or divergence of the normalizing change under the analyticity assumption on $\widehat{u}$ is a central question in this theory. Nevertheless, Eliasson drew attention to another (more difficult) problem: the convergence/divergence of the normal form produced during normalization. In the Hamiltonian case, in the absence of resonances, there is a statement that if the normalization converges, then the system is locally completely integrable. Various versions of the converse statement were proposed in \cite{Vey, Ito, Eliasson, Kappeler}.

In this paper we prove the Siegel-Brjuno theorem on the normalization of a vector field, which generalizes Siegel’s theorem. A precise formulation of the Siegel–Brjuno theorem will be given in Section \ref{formulationZigel}. Siegel's theorem (see \cite{Arnold,VanLiChoy,Bernard}) asserts that if the eigenvalues \eqref{eigenvalues} satisfy the Diophantine condition, i.e., there exist $C_0 >0$ and $\mu>0$ such that for each $\bkk \in \mZ_{\diamd}^n$,
\begin{equation*}
	|\langle \lambda , \bkk \rangle - \lambda_j| \geq \frac{C_0}{|\bkk|^{\mu}},\qquad 1\leq j \leq n,
\end{equation*}
then the normal form is trivial, and equation \eqref{ODEinitial} can be reduced to $\dot{z} = \Lambda z$ by an analytic change of variables.

The classical methods of proof for this theorem consist either in studying the formal series defining the conjugacy, or in using an iterative approach. The first method was used by Siegel in \cite{Sigel}. Later, Siegel’s theorem was generalized (Diophantine conditions were replaced by Brjuno conditions) in \cite{Bruno}. The second method was used by Rüssmann \cite{Russman} and was proposed (under Diophantine conditions on the eigenvalues) in the books \cite{Arnold,CaresonGamelin}. A shorter proof of the Brjuno theorem is given in \cite{Bernard}.

In the Hamiltonian case, if one knows in advance that the normal form is trivial, i.e., has the form $N = \sum_{j=1}^n w_j z_j \overline{z_j}$, then, according to the Bruno–Rüssmann theorem \cite{Bruno,Russman} (see also \cite{Stolovich}), under the Brjuno condition on the frequencies $\omega$, the normalization is convergent.

The normalization process is constructed as a composition consisting of infinitely many changes of variables that normalize the system degree by degree; see \cite{Arnold}.

D.\,V. Treschev proposed a new approach to the normalization of Hamiltonian systems based on the method of continuous averaging (see \cite{TreschevZub}). This approach was formulated in \cite{TreschevNorm}. In \cite{TreschevNorm} the space of all Hamiltonian vector fields near a nonresonant elliptic singular point is considered. On this space a certain differential equation is defined that generates a flow moving the Hamiltonian system to its normal form. The shift along this flow corresponds to a canonical change of coordinates. Thus, a continuous normalization process is constructed.

In the present paper we propose a new approach to the normalization of a system of differential equations near a singular point, based on the method of continuous averaging; see \cite{TreschevZub}. Let $\calF$ be the space (in fact, a Lie algebra) of formal series:
\begin{equation}\label{formalseries}
	u(z) = \sum_{\bkk \in \mZ_{\diamd}^n} U_{\bkk} \bzz^{\bkk}, \qquad u \in \calF.
\end{equation}

We endow $\calF$ with the Tikhonov topology, i.e., a sequence $\{u^{(j)}\}_{j=1}^{\infty}$ is said to converge if, for any $\bkk \in \mZ^n _{\diamd}$ and any $m=1,\ldots,n$, the sequence of coefficients $\{(U_{\bkk} ^m)^{(j)}\}_{j=1}^{\infty}$ converges.

The following results are obtained in this work:

1. In Section 2 we construct a normalization flow $\psi^{\delta}\colon \calF \to \calF$, $\delta \in \mR_{+}$. Consider the commutator $[\, , \,]$ on the space of vector fields, so that for any $u,v \in \calF$:
\begin{equation*}
	[u,v]_i = \sum_{j=1}^n \bigg(u_j\frac{\pa v_i}{ \pa z_j}  - v_j\frac{\pa u_i}{\pa z_j}\bigg).
\end{equation*}
The flow $\psi^{\delta}$ is defined by a certain differential equation on $\calF$:
\begin{equation}\label{certainODEonF}
	\pa_{\delta} u = -[\xi u,\Lambda z + u],\qquad u \big|_{\delta=0} = \widehat{u}.
\end{equation}
In equation \eqref{certainODEonF}, $\xi$ is a certain linear operator on $\calF$, defined in Section~2.

2. Section 3 studies properties of the normalization flow $\psi^{\delta}$. In particular, we prove that the subalgebra of Hamiltonian vector fields is invariant under shifts along the flow $\psi^{\delta}$. Moreover, in this section we show that the flow $\psi^{\delta}$ preserves a discrete symmetry of the vector field.

3. In Section 4 we prove that the shift along the flow $\psi^{\delta}$
\begin{equation*}
	\Lambda z + \widehat{u}_{\diamd} \mapsto \Lambda z + \psi^{\delta} (\widehat{u}_{\diamd}), \qquad \widehat{u}_{\diamd} \in \calF
\end{equation*}
moves the vector field $\widehat{u}_{\diamd}$ to the corresponding normal form as $\delta \to +\infty$. The space of normal vector fields $\calN \subset \calF$ is defined by
\begin{equation*}
	\calN = \{u \in \calF \colon   U_{\bkk} ^s \neq 0 \text{ implies } \langle \lambda,\bkk_s \rangle =0 \}.
\end{equation*}
Thus, $\lim_{\delta\to +\infty} \psi^{\delta}u \in \calN$, where the limit is taken with respect to the Tikhonov topology on $\calF$.

4. Section 5 deals with the analytic aspect of the normalization theory. We restrict the flow $\psi^{\delta}$ to the space of analytic functions $\calA \subset \calF$:
\begin{equation*}
	\calA = \{u \in \calF\colon \exists\, a, b \text{ such that } |U_{\bkk}^m| \leq a e^{b |\bkk|} \text{ for all } \bkk \in \mZ^{n}_{\diamd}\text{ and } m =1,\ldots,n  \}.
\end{equation*}
Next, in $\calA$ we introduce a more natural topology, substantially stronger than the Tikhonov topology. We have $\calA = \cup_{\rho>0} \calA^{\rho}$, where $\calA^{\rho}$ is a Banach space with norm
\begin{equation*}
	\| u\|_{\rho} = \max_{m \in \{1,\ldots,n\}} \sup_{\bzz\in \calD_{\rho}} |u^m(\bzz)|,\qquad \calD_{\rho} = \{\bzz \in \mC^n\colon |z_j| < \rho,\quad j=1,\ldots,n\}.
\end{equation*}
In this section we show that for any vector field $\widehat{u} \in \calA$ we have $\psi^{\delta}(\widehat{u}) \in \calA$ for every $\delta\geq 0$. However, as $\delta$ increases, the polydisk of analyticity shrinks. An estimate of the width of the domain of analyticity is provided; it is of order $1/\delta$.

5. Section 6 is devoted to the proof of the Siegel–Brjuno theorem on the convergence of normalization. The proof is based on the method of continuous averaging (see \cite{TreschevZub}). In addition, to prove this theorem we use the so-called method of superconvergence. The proof is based on ideas presented in \cite{Tres2}.

6. Section 7 contains technical material used to prove the theorem on the domain of analyticity in Section 4 and the Siegel–Brjuno theorem.

The author is grateful to D.\,V. Treschev for numerous discussions and valuable remarks. The author also thanks the referees and the editorial board for their helpful comments.

\section{Continuous Averaging} \label{averaging}

Here we construct a flow $\psi^{\delta},\, \delta \in \mR$ on $\calF$ that asymptotically (in the Tikhonov topology as $\delta \to +\infty$) carries a vector field $\widehat{u} \in \calF$ to the corresponding normal vector field $\psi^{+\infty} \widehat{u} \in \calN$.

To construct this flow, we use the method of continuous averaging, as formulated in \cite{TreschevZub}.

Along with system \eqref{ODEinitial} we also consider the system
\begin{equation}\label{ODEf}
	Z' = f(Z, \delta), \qquad Z(z,\delta) \big|_{\delta=0} = z, \qquad  f(Z) = O(|Z|^2),\qquad Z' = \frac{d}{d\, \delta} Z,
\end{equation}
where $f(Z)=O(|Z|^2)$ is some function whose series starts from quadratic terms. We will define the function $f(Z)$ below. We will denote the variable $Z$ by $z$. The solution of system \eqref{ODEf} defines a family of changes of variables: $\bzz \mapsto \bzz_\delta.$ Substituting this change of variables into \eqref{ODEinitial}, we obtain
\begin{equation}\label{ODEu}
	\dot{z} = u(z,\delta),\qquad u(z,\delta) = \Lambda z + \sum_{\bkk\in \mZ^n_{\diamd}} U_{\bkk}(\delta) \bzz^{\bkk}. 
\end{equation}

Differentiating \eqref{ODEu} with respect to $\delta$ and \eqref{ODEf} with respect to $t$, and subtracting one equation from the other, we obtain
\begin{equation*}
	\pa_{\delta} u = -[f,u],\qquad u \big|_{\delta=0} = \widehat{u}.
\end{equation*}

The main idea of continuous averaging is to take $f$ in the form $\xi u$, where $\xi$ is a linear operator on $\mathcal{F}$, to be defined below. Thus we arrive at
\begin{equation}\label{ODExiudiamd}
	\pa_{\delta}u = -[\xi u, u],\qquad u \big|_{\delta=0} = \widehat{u}.
\end{equation}

The choice of the operator $\xi$ depends on the target normal form we wish to obtain. In the following sections we show that the solution to \eqref{ODExiudiamd} with the chosen operator $\xi$ exists and is unique.

\subsection{The operator $\xi$}
For each $u \in \calF$, we put
\begin{equation}\label{oepratorxi}
	(\xi u)^m = - \sum_{\bkk \in \mZ^n_{\diamd},\; \langle \lambda, \bkk_m \rangle \neq 0 }	 e^{-i \arg{\langle \lambda,\bkk_m \rangle}}\, U_{\bkk}^m \bzz^{\bkk}, \qquad m \in \{1,\ldots,n\}.
\end{equation}

We also put
\begin{equation*}
	u = \Lambda z + u_0 + u_{\ast},\qquad u^m_0 = \sum_{\bkk \in \mZ^n _{\diamd},\; \langle \lambda,\bkk_m \rangle = 0} U_{\bkk}^m \bzz^{\bkk},\qquad u_{\ast}^m = \sum_{\bkk \in \mZ^n _{\diamd},\; \langle \lambda,\bkk_m \rangle \neq 0} U_{\bkk}^m \bzz^{\bkk}.
\end{equation*}

Thus we obtain a more detailed form of \eqref{ODExiudiamd}:
\begin{equation}\label{detailformODE}
	\begin{split}
		&\qquad\qquad \pa_{\delta} u = l(u) + v_0 (u) + v_{\ast} (u),\qquad u\big|_{\delta=0} = \widehat{u},\\
		&l (u) = -[\xi u,\Lambda z],\qquad v_0 = -[\xi u, u_0],\qquad v_{\ast}(u) = - [\xi u, u_{\ast}].
	\end{split}
\end{equation}

The informal explanation for the choice \eqref{oepratorxi} is that, after removing the terms $v_0$ and $v_{\ast}$ in \eqref{detailformODE}, we arrive at
\begin{equation}\label{nonlinearODE}
	\begin{split}
		\pa_{\delta} u &= l(u),\qquad u\big|_{\delta=0} = \widehat{u},\\
		l^m(u) &= -\sum_{\bkk \in \mZ^n _{\diamd},\; \langle \lambda,\bkk_m \rangle \neq 0} |\langle \lambda,\bkk_m \rangle|\, U_{\bkk}^m \bzz^{\bkk}.
	\end{split}
\end{equation}

Equation \eqref{nonlinearODE} is easily solved:
\begin{equation*}
	U_{\bkk}^m = e^{- |\langle \lambda,\bkk_m \rangle| \delta}\, \widehat{U}^m_{\bkk},\qquad \langle \lambda,\bkk_m \rangle \neq 0.
\end{equation*}

We see that as $\delta \to +\infty$, the flow generated by \eqref{nonlinearODE} carries the vector field to the subspace of normal forms $\calN$.

\subsection{Computation of $v_0$ and $v_{\ast}$}
For each $U = \sum_{\bkk \in \mZ^n _{\diamd}} U_{\bkk} \bzz^{\bkk}$, define $p^m_{\bss}(U) = U^m_{\bss}$, $m = 1,\ldots,n$. Make the change of variables
\begin{equation}\label{changeofvariables}
	U_{\bkk} ^{m} = \calU_{\bkk}^m\, e^{-|\langle \lambda,\bkk_m \rangle|\delta}.
\end{equation}

Substituting \eqref{changeofvariables} into \eqref{detailformODE}, we obtain
\begin{equation}\label{reducedODE}
	\pa_{\delta}\calU_{\bdd}^m =\mathbf{v}^m_{0,\bdd} + \mathbf{v}^m_{\ast, \bdd},\qquad \calU_{\bdd} \big|_{\delta=0} = \widehat{U}_{\bdd}, 
\end{equation}
where $\mathbf{v}^m_{0,\bdd}$ and $\mathbf{v}^m_{\ast, \bdd}$ are given by the following formulas.

If $\langle \lambda,\bdd_m \rangle \neq 0$:
\begin{equation}\label{detailedformv0}
	\begin{split}
		\mathbf{v}^m_{0,\bdd}=   e^{-i \arg{\langle \lambda,\bdd_m \rangle}} \sum_{p=1}^n \Bigg( \sum_{\substack{
				\bkk + \bss - \bdd = e_p,\\
				\langle \lambda,\bkk_m \rangle=0,\\
				\langle \lambda,\bss_p \rangle\neq 0}} k_p \calU_{\bkk}^m \calU_{\bss}^p   -  \sum_{\substack{
				\bkk + \bss - \bdd = e_p,\\
				\langle \lambda,\bss_p \rangle=0,\\ 	\langle \lambda,\bkk_m \rangle\neq 0}} k_p  \calU_{\bkk}^m \calU_{\bss}^p \Bigg ).
	\end{split}
\end{equation}

If $\langle \lambda,\bdd_m \rangle = 0$:
\begin{equation} \label{eqv0mzero}
	\mathbf{v}^m_{0,\bdd} = 0.
\end{equation}

The second term in \eqref{reducedODE} has the form
\begin{equation}\label{detailedformvstar}
	\mathbf{v}^m_{\ast,\bdd}= \sum_{p=1}^n \sum_{\substack{\bkk + \bss -\bdd = e_p,\\\langle \lambda,\bss_p \rangle\neq0,\\
			\langle \lambda,\bkk_m \rangle\neq0}}  k_p\,\calU_{\bkk}^m \calU_{\bss}^p\, e^{T_{\bkk_m,\bss_p }\delta} \Big(e^{-i \arg{\langle \lambda,\bkk_m \rangle}} - e^{-i \arg{\langle \lambda,\bss_p \rangle}}\Big).
\end{equation}

In \eqref{detailedformvstar} we denote
\begin{equation}\label{Tksdenoting}
	T_{\bkk,\bss}=|\langle \lambda,\bkk \rangle + \langle \lambda,\bss \rangle| - |\langle \lambda,\bkk \rangle| - |\langle \lambda,\bss \rangle| \leq 0.
\end{equation}

First, we prove \eqref{detailedformv0}. Suppose $\langle \lambda,\bdd_m \rangle \neq 0$. Then, by a direct computation of the commutator (see \eqref{detailformODE}), we obtain
\begin{equation}\label{detailedformv00}
	\begin{split}
		\mathbf{v}^m_{0,\bdd}=   \sum_{p=1}^n \Bigg( \sum_{\substack{
				\bkk + \bss - \bdd = e_p,\\
				\langle \lambda,\bkk_m \rangle=0,\\
					\langle \lambda,\bss_p \rangle\neq 0}} k_p \calU_{\bkk}^m e^{-i \arg{\langle \lambda,\bss_p \rangle}}\calU_{\bss}^p   -  \sum_{\substack{
				\bkk + \bss - \bdd = e_p,\\
				\langle \lambda,\bss_p \rangle=0,\\ 	\langle \lambda,\bkk_m \rangle\neq 0}} k_p e^{-i \arg{\langle \lambda,\bkk_m \rangle}} \calU_{\bkk}^m \calU_{\bss}^p \Bigg ).
	\end{split}
\end{equation}

The first sum in \eqref{detailedformv00} is taken over $\bkk, \bss, \bdd$ such that $\langle \lambda,\bkk_m \rangle=0$ and $\langle \lambda,\bss_p \rangle \neq 0$. Using
\begin{equation} \label{eq12}	
		\langle \lambda,\bkk_m \rangle + \langle \lambda,\bss_p \rangle = \langle \lambda,\bdd_m \rangle,
\end{equation}
we obtain $0\neq \langle \lambda,\bss_p \rangle = \langle \lambda,\bdd_m \rangle$. Similarly, for the second sum in \eqref{detailedformv00}, where $\langle \lambda,\bkk_m \rangle\neq 0$ and $\langle \lambda,\bss_p \rangle = 0$, we get $0\neq \langle \lambda,\bkk_m \rangle = \langle \lambda,\bdd_m \rangle$. Hence
\begin{equation}\label{detailedformv000}
	\begin{split}
		\mathbf{v}^m_{0,\bdd}=   e^{-i \arg{\langle \lambda,\bdd_m \rangle}} \sum_{p=1}^n \Bigg( \sum_{\substack{
				\bkk + \bss - \bdd = e_p,\\
				\langle \lambda,\bkk_m \rangle=0,\\
				\langle \lambda,\bss_p \rangle\neq 0}} k_p \calU_{\bkk}^m \calU_{\bss}^p   -  \sum_{\substack{
				\bkk + \bss - \bdd = e_p,\\
				\langle \lambda,\bss_p \rangle=0,\\ 	\langle \lambda,\bkk_m \rangle\neq 0}} k_p  \calU_{\bkk}^m \calU_{\bss}^p \Bigg ), \quad \langle \lambda,\bdd_m \rangle \neq 0.
	\end{split}
\end{equation}

If $\langle \lambda,\bdd_m \rangle = 0$, then $\mathbf{v}^m_{0,\bdd} =0$. Indeed, looking at the first sum in \eqref{detailedformv00} together with \eqref{eq12}, each term would satisfy $\langle \lambda,\bkk_m \rangle =0$ and $\langle \lambda,\bss_p \rangle = \langle \lambda,\bdd_m \rangle=0$, which contradicts $\langle \lambda,\bss_p \rangle\neq 0$ in that sum. Similarly, the second sum in \eqref{detailedformv00} vanishes when $\langle \lambda,\bdd_m \rangle = 0$.

\section{Properties of the flow $\psi^{\delta}$}
\subsection{Invariance of the flow $\psi^{\delta}$ with respect to certain subspaces}
Consider
\begin{equation*}
	\begin{split}
		\calB_M &= \{\bkk \in \mZ^n \colon |\bkk| \leq M\},\\
		\calS_M &= \{\bkk \in \mZ^n\colon \text{Re}\,\langle \lambda,\bkk  \rangle \leq M\},\qquad M\geq 0. 
	\end{split}
\end{equation*}

Consider the following subspaces:
\begin{equation*}
	\begin{split}
		\calF_{\calB_M} &= \{u\in \calF \colon U_{\bkk} ^m = 0 \text{ for every } m\in \{1,\ldots,n\},\; \bkk \in  \calB_M\},\\
		\calF_{\calS_M} &= \{u\in \calF \colon U_{\bkk} ^m = 0 \text{ for every } m\in \{1,\ldots,n\},\; \bkk \in  \calS_M\}.
	\end{split}
\end{equation*}

\begin{lemma}
The flow $\psi^{\delta}$ is invariant with respect to $\calF_{\calB_M}$ and $\calF_{\calS_M}$, i.e.:
	
1. If $\widehat{u} \in \calF_{\calB_M}$, then $\psi^{\delta} \widehat{u} \in \calF_{\calB_M}$ for every $\delta \geq 0$.
	
2. If $\widehat{u} \in \calF_{\calS_M}$, then $\psi^{\delta} \widehat{u} \in \calF_{\calS_M}$ for every $\delta \geq 0$.
\end{lemma}

\begin{proof}
We prove the first claim by induction on $\bdd$. Let $|\bdd|=2 \leq M$. Looking at \eqref{reducedODE} and using $\bkk + \bss - \bdd = e_p$, we obtain
\begin{equation}\label{property001}
	\pa_{\delta} \calU_{\bdd}^m = 0,\qquad \calU_{\bdd} ^m\big|_{\delta=0} = 0.
\end{equation}

Hence, by \eqref{property001} we have $\calU_{\bdd} ^m (\delta) = 0$ for every $\delta \geq 0$. Suppose now that for every $\bdd$ with $|\bdd| \leq N-1$ the first statement of the lemma holds. Take $\bdd$ with $|\bdd| = N \leq M$. From $\bkk + \bss - \bdd = e_p$ we get $|\bkk| \leq N-1$, $|\bss| \leq N-1$. In this case \eqref{reducedODE} takes the form
\begin{equation} \label{property002}
	\pa_{\delta} \calU_{\bdd}^m = 0,\qquad \calU_{\bdd} ^m\big|_{\delta=0} = 0.
\end{equation}

From \eqref{property002} it follows that $\calU_{\bdd} ^m (\delta) = 0$ for every $\delta \geq 0$. Therefore $\psi^{\delta} \widehat{u} \in \calF_{\calB_M}$.

The second claim is proved similarly by induction on $\bdd$.
\end{proof}

\subsection{Symmetries and the flow $\psi^{\delta}$}

Consider the symmetries
\begin{equation*}
	\sigma\colon (z_1, z_2,\ldots,z_n) \mapsto (z_{\sigma(1)},z_{\sigma(2)},\ldots, z_{\sigma(n)}),
\end{equation*}
where $\sigma$ is a permutation of the set $\{1,\ldots,n\}$. We call the vector field $\Lambda z + u(\bzz)$, $u \in \calF$ $\sigma$-invariant if $\sigma \circ(\Lambda + u) = \Lambda + u$.

\begin{lemma}
Assume $\Lambda z + \widehat{u}$ is $\sigma$-invariant. Then, for every $\delta \geq 0$, the shift along the flow $\Lambda z + \psi^{\delta} \widehat{u}$ is $\sigma$-invariant.
\end{lemma}

\begin{proof}
For the left-hand side of \eqref{ODExiudiamd} we have
\begin{equation*}
	\sigma \circ \pa_{\delta} u = \pa_{\delta}( \sigma \circ u).
\end{equation*}

For the right-hand side of \eqref{ODExiudiamd} we get
\begin{equation*}
	-\sigma \circ [\xi u, \Lambda z + u] = - [\sigma \circ (\xi u), \sigma \circ (\Lambda z + u)].
\end{equation*}

A direct computation shows that $\sigma \circ (\xi u) = \xi (\sigma \circ u)$. Therefore,
\begin{equation*}
	\pa_{\delta}( \sigma \circ (\Lambda z + u) )= - [\xi (\sigma \circ (\Lambda z + u)), \sigma \circ (\Lambda z + u)], \qquad \sigma \circ u \big|_{\delta = 0} = \widehat{u}.
\end{equation*}

Hence $\sigma \circ (\Lambda z + u)$ and $\Lambda z + u$ are solutions of \eqref{ODExiudiamd} with the same initial condition; thus $\sigma \circ (\Lambda z + u) = \Lambda z + u$.
\end{proof}

\subsection{Hamiltonian structure and the flow $\psi^{\delta}$}

Consider the Hamiltonian system
\begin{equation}\label{hamsystem}
	\dot{x} = J \triangledown \widehat{H}(x), \qquad x \in \mR^{2n},
\end{equation}
where $J$ is the standard symplectic matrix, and $\widehat{H}(x) = \sum_{|\bkk| \geq 3} \widehat{h}_{\bkk} \bxx ^{\bkk}$ is the Hamiltonian of the system. Assume that $\widehat{H}_2(x) = (Ax,x)$, with
\begin{equation}\label{eignvalues}
	A= \text{diag}(\lambda_1,\ldots,\lambda_n, -\overline{\lambda}_1,\ldots,-\overline{\lambda}_n), \text{ where } \lambda_1,\ldots,\lambda_n \in \mC.
\end{equation}

According to the continuous averaging method, the normalizing change of variables (see \eqref{ODEf}) satisfies
\begin{equation}\label{normchange}
	x' = \xi J \triangledown H(x,\delta).
\end{equation}

Let $\calH \subset \calF$ be the Lie subalgebra of Hamiltonian vector fields.

\begin{lemma} \label{haminvariant}
Consider the Hamiltonian system \eqref{hamsystem}. Suppose its linear part is the diagonal matrix defined in \eqref{eignvalues}. Then for every $\delta \geq 0$ we have $\psi^{\delta}J \triangledown \widehat{H}(x) \in \calH$.
\end{lemma}

\begin{proof}
To prove Lemma~\ref{haminvariant}, we show that
\begin{equation*}
	\xi J \triangledown H(x) = J\triangledown\big( \theta H(x)\big), \qquad \theta\colon  H(x) \mapsto \sum_{ \langle \lambda, \bkk \rangle \neq 0 } e^{-i\arg\langle \lambda, \bkk \rangle } h_{\bkk} \bxx ^ {\bkk}.
\end{equation*}

Indeed, for any $s\leq n$,
\begin{equation*}
	\begin{split}
		\big(\xi J \triangledown H(x) \big)^s&= \bigg(\xi \sum_{\bkk \in \mZ^{2n}_{\diamd}} k_{n+s}\, h_{\bkk}\,\bxx^{\bkk - e_{n+s}}\bigg)^s \\
		&= - \sum_{\bkk:\, \langle \lambda,\bkk\rangle \neq 0} e^{-i \arg\langle\lambda,\bkk\rangle}\, k_{n+s}\, h_{\bkk}\, \bxx^{\bkk - e_{n+s}}
		= \frac{\pa}{\pa x_{n+s}}\, \theta H(x).
	\end{split}
\end{equation*}

Similarly, for any $s \leq n$,
\begin{equation*}
	\big(\xi J \triangledown H(x) \big)^{s+n} = - \frac{\pa}{\pa x_s}\, \theta H(x).
\end{equation*}

Consequently, the normalizing change satisfies the Hamiltonian system with Hamiltonian $\theta H(x,\delta)$:
\begin{equation}\label{hamnormchange}
	x' = J\triangledown\big( \theta H(x,\delta)\big).
\end{equation}

It follows that the normalizing changes of variables satisfying \eqref{hamnormchange} are canonical; in particular, they preserve the Hamiltonian structure.
\end{proof}

\section{Formal Aspect}

In this section we study the formal aspect of the theory. To prove the main result of this section, we first prove the following lemma.

\begin{lemma} \label{lemmaaboutUd}
	Assume that $\widehat{u}_{\diamd} \in \calF$. Then for every $\bdd \in \mZ^n_{\diamd}$ the solution of the Cauchy problem \eqref{reducedODE} has the form
	\begin{equation}\label{thesolutionofodereduced}
		\begin{split}
			\calU^m _{\bdd}(\delta) &= \widehat{U}_{\bdd}^m + \calQ_{\bdd} ^m (\widehat{u}_{\diamd},\delta), \quad \text{if } \langle \lambda , \bdd_m \rangle = 0,\\
			\calU^m _{\bdd}(\delta) &= \widehat{U}_{\bdd}^m + \calP_{\bdd} ^m (\widehat{u}_{\diamd},\delta), \quad \text{if } \langle \lambda , \bdd_m \rangle \neq 0, \qquad m \in \{1,\ldots,n\},
		\end{split} 
	\end{equation}
	where $\calQ_{\bdd}^m$ and $\calP_{\bdd}^m$ are polynomials in the variables $\widehat{U}_{\bm}^m$ with $|\bm| < |\bdd|$, whose coefficients are finite linear combinations of terms of the form $\delta^s e^{-\nu \delta}$, $s\in \mZ_{+}$, $\nu \ge 0$. Moreover, in the polynomial $\calQ_{\bdd}^m$ such terms $\delta^s e^{-\nu \delta}$ satisfy the condition: if $\nu = 0$, then $s=0$.
\end{lemma}

\begin{proof}
	We prove by induction. Base step. If $|\bdd|=2$, then \eqref{reducedODE} reads
	\begin{equation*}
		\pa_{\delta} \calU_{\bdd} ^m =0,\qquad |\bdd|=2.
	\end{equation*}
	Hence, for $|\bdd| = 2$ we have $\calU_{\bdd}^m = \widehat{U}_{\bdd}^m$. Assume now that \eqref{thesolutionofodereduced} holds for all vectors $\bdd \in \mZ^n _{\diamd}$ with $|\bdd| < K$ (induction hypothesis). Take $\bdd \in \mZ^n _{\diamd}$ with $|\bdd|=K$ and $\langle \lambda, \bdd_m \rangle \neq 0$. Integrating \eqref{reducedODE} we obtain
	\begin{equation*}
		\calU_{\bdd}^m = \widehat{U}_{\bdd}^m + I_1 + I_2,\qquad 
		I_1 = \int_{0} ^ {\delta} \mathbf{v}^m _{0,\bdd} (\lambda)\, d\lambda,\qquad 
		I_2 = \int_{0} ^ {\delta} \mathbf{v}^m _{\ast,\bdd} (\lambda)\, d\lambda.
	\end{equation*}
	Using \eqref{detailedformv0}, \eqref{detailedformvstar}, and the induction hypothesis yields \eqref{thesolutionofodereduced}. If $\langle \lambda , \bdd_m \rangle = 0$, then $\mathbf{v}_{0, \bdd} = 0$ (see \eqref{eqv0mzero}) and
	\begin{equation*}
		\calU_{\bdd}^m = \widehat{U}_{\bdd}^m  + I_2,\qquad 
		I_2 = \int_{0} ^ {\delta} \mathbf{v}^{m} _{\ast,\bdd} (\lambda)\, d\lambda.
	\end{equation*}
	From \eqref{detailedformvstar} it follows that the condition $\nu=0$ forces $s=0$.
\end{proof}

\begin{theorem}
	The limit $\lim_{\delta\to +\infty} \psi^{\delta}u$ exists and lies in $\calN$, where this limit is taken with respect to the Tikhonov topology on $\calF$.
\end{theorem}

\begin{proof}
	The proof is based on Lemma~\ref{lemmaaboutUd}. Make the change of variables \eqref{changeofvariables} when $\langle \lambda , \bdd_m \rangle \neq 0$. Then, as $\delta \to \infty$ we have $U_{\bdd}^m \to 0$. If $\langle \lambda , \bdd_m \rangle = 0$, the form of the polynomial $\calQ_{\bdd} ^m$ gives the claim, since it contains no terms of the type $\delta^s$ with $s>0$.
\end{proof}

\section{Analytic Aspect}

\subsection{Domain of convergence}
\begin{theorem} \label{thabotradius}
	Assume that $\widehat{u} \in \calA^{\rho} \cap \calF$. Then for any $\delta \geq 0$ and $j \in \{1,\ldots,n\}$,
	\begin{equation*}
		u =\psi^{\delta}(\widehat{u})= \sum_{\bkk} \calU_{\bkk} \bzz^{\bkk} \in \calA^{g(\rho,\delta)} \cap \calF,
	\end{equation*}
	where
	\begin{equation*}
		g(\rho,\delta) \geq \frac{\rho^2}{2n(\rho+ 8\|\widehat{u}\|_{\rho} n \delta)}.
	\end{equation*}
	Moreover, for any $j \in \{1,\ldots,n\}$,
	\begin{equation*}
		|u^j| \leq \frac{\rho^{2} + 4n\delta \| \widehat{u} \|_{\rho}}{4n\delta \, (\rho + 4n\delta \|\widehat{u} \|_{\rho})}
		+
		\frac{\rho^{2}}{8n^{2}\delta \, (\rho + 8n\delta \|\widehat{u} \|_{\rho})},
	\end{equation*}
	in the domain
	\begin{equation*}
	 \left\{z \in \mC^n \colon |z_1 + \ldots + z_n| \leq \frac{\rho^2}{2n(\rho + 8 \| \widehat{u}\|_{\rho}  n \delta)}\right\}.
	\end{equation*}
\end{theorem}

\begin{proof}
	To prove the theorem, we use the majorant method. For each $m\in \{1,\ldots,n\}$, consider the majorant equation for \eqref{reducedODE}:
	\begin{equation}\label{majornatequation}
		\begin{split}
			\pa_{\delta} \mathbf{U}_{\bdd}^m &= \mathbf{V}_{\bdd}^m, \qquad \mathbf{V}_{\bdd}^m \big|_{\delta=0} = \widehat{\mathbf{U}}_{\bdd}^m,\\
			\mathbf{V}_{\bdd}^m &= 4 \sum_{p=1}^n \sum_{\bkk +\bss-\bdd=e_p} \mathbf{U}^m_{\bss}\, s_p\, \mathbf{U}_{\bkk}^p.
		\end{split}
	\end{equation}
	To obtain \eqref{majornatequation}, replace the minus signs in $\mathbf{v}_{0,\bdd}$ and $\mathbf{v}_{\ast,\bdd}$ by plus signs, remove the exponential factors, and add certain positive terms. Assume that in \eqref{majornatequation} we have $|\widehat{\calU}_{\bdd}^m | \leq \widehat{\mathbf{U}}_{\bdd}^m$ for every $\bdd \in \mZ^n _{\diamd}$. The system \eqref{majornatequation} can be written more compactly as
	\begin{equation*}
		\pa_{\delta} \mathbf{U}^m = 4 \sum_{p=1}^n\mathbf{U}^p \pa_{z_p} \mathbf{U}^m, \qquad \mathbf{U}^m \big|_{\delta=0} = f(z_1,\ldots,z_n).
	\end{equation*}
	We will choose $f(z)$ below. We take $\mathbf{U}$ so that for each $m \in \{1,\ldots,n\}$ and every $\bdd \in \mZ^n_{\diamd}$ we have $|\mathbf{U}^m _{\bdd}| \leq \mathbf{U}_{\bdd}$. Then the majorant system has the form
	\begin{equation}\label{shorterformmajor}
		\pa_{\delta} \mathbf{U} = 4 \sum_{p=1}^n\mathbf{U} \pa_{z_p} \mathbf{U}, \qquad \mathbf{U} \big|_{\delta=0} = f(z_1,\ldots,z_n).
	\end{equation}
	Majorize $\mathbf{U}$ by $\mathbf{U}(z,\delta) \ll F(\zeta,\delta)$, where $\zeta = \sum_{j=1}^n z_j$. The following Burgers equation majorizes \eqref{shorterformmajor}:
	\begin{equation}\label{theburgerdif}
		\pa_{\delta} F= 4 n F \pa_{\zeta} F,\qquad F\big|_{\delta=0} = f(\zeta).
	\end{equation}
	The solution $F=F(\zeta,t)$ of \eqref{theburgerdif} satisfies
	\begin{equation}\label{justeqinth}
		F=f(\zeta+4ntF).
	\end{equation}
	Using Lemma~\ref{thefunctf}, choose
	\begin{equation} \label{eq:choiceoff}
		f(\zeta) = a \zeta^2/(b-\zeta),\qquad a = \frac{\|\widehat{u} \|_{\rho}}{\rho^2}\,\rho,\quad b = \rho.
	\end{equation} 
	Setting $\tau = 4n\delta$, we obtain
	\begin{equation*}
		F= \frac{a(\zeta+\tau F)^2}{b - \zeta - \tau F}.
	\end{equation*}
	Its solution is
	\begin{equation} \label{functionF}
		F = \frac{b-\zeta-2 a \tau \zeta - \sqrt{(b-\zeta-2a\tau\zeta)^2 - 4 a \tau \zeta^2 (1+a \tau)}}{2\tau(1+a\tau)}.
	\end{equation}
	The branch points of $F(\zeta)$ (zeros of the radicand in \eqref{functionF}) are
	\begin{equation*}
		\zeta_{1,2} = \frac{b}{1+2a\tau \pm 2\sqrt{ a \tau(1+a\tau)}}.
	\end{equation*}  
	Hence $F=F(\zeta)$ is analytic whenever
	\begin{equation} \label{eq: radius}
	\begin{split}
		|\zeta| \leq  \min\big(\zeta_1, \zeta_2\big) = \frac{b}{1+2a\tau + 2\sqrt{a \tau(1+a\tau)}}.
	\end{split}
	\end{equation}
	Consequently, $F(\zeta)$ is analytic in the smaller disc
	\begin{equation*}
		|\zeta| \leq d(\tau)=\frac{b}{2(1+2a\tau)} < \min\big(\zeta_1, \zeta_2\big).
	\end{equation*}
	The radius of the polydisc on which $F(\zeta)=F(z_1+\cdots+z_n)$ is analytic equals $R(\tau)=d(\tau)/n$. Indeed, if $\mathbf{z} \in \mathcal{D}_{R(\tau)}$, then
	\begin{equation*}
	|\zeta| = |z_1 + \ldots + z_n| \leq |z_1| + \ldots + |z_n| \leq d(\tau).	
	\end{equation*}
	Since $F(\zeta)$ majorizes $u = \psi^{\delta} (\widehat{u})$, we obtain a lower bound for the radius of the polydisc of convergence for $u = \psi^{\delta} (\widehat{u})$:
	\begin{equation*}
		g(\rho,\delta) \geq R(\tau) = \frac{b}{2n(1+2a\tau)} = \frac{\rho^2}{2n (\rho + 8 \| \widehat{u}\|_{\rho}n \delta)}.
	\end{equation*}

	If
	\begin{equation*}
		|\zeta| \leq  \frac{b}{2 n (1+2a\tau)},
	\end{equation*}
	then using \eqref{functionF} we get
	\begin{equation*}
	\begin{split}
		|u^j| \leq |F| &\leq \frac{b+ |\zeta| + 2a\tau|\zeta| + \sqrt{|\zeta|^2 + 2 |\zeta|b + b^2 + 4ab |\zeta| \tau}}{2 \tau (a\tau +1)} \\
		&\leq \frac{b+ |\zeta| + 2a\tau|\zeta| + \sqrt{(|\zeta| +b + 2a\tau)^2}}{2 \tau (a\tau +1)} \\ 
		&\leq \frac{b+a \tau}{\tau(1+a\tau)} + \frac{b}{2  n \tau (1+ 2 a \tau)}, \qquad u^j = \big(\psi^{\delta}(\widehat{u})\big)^j,\quad j \in \{1,\ldots,n\}.
	\end{split}
	\end{equation*}
\end{proof}

Under certain conditions on the eigenvalues of the matrix $\Lambda$ (e.g. when they lie in the Siegel domain, in the presence of resonances), the changes leading to the normal form often diverge (see \cite{Arnold}, §24). Thus, in general, one should not expect $\psi^{+\infty} \widehat{u}$ to belong to $\calA \cap \calF$.

\section{On the convergence of normalization} \label{formulationZigel}

\subsection{Statement of the Siegel–Bruno theorem}

In this section we prove the Siegel–Bruno theorem on the convergence of normalization by the method of continuous averaging. To state the main theorem of this section, we first introduce some definitions.

For each $s \in \mN$, set
\begin{equation} \label{def: Omegasdef}
	\Omega_s = \max\bigg\{\frac{1}{|\langle \lambda, \bkk \rangle |} \colon \bkk \in \mZ^n \setminus \bL_{\lambda} , \; 0< |\bkk|\leq s\bigg\},
\end{equation}
where $\mathbf{L}_{\lambda}$ is defined in \eqref{eq: defofLlamvda}.
\begin{definition}
A sequence $\{a_j\}_{j \in \mZ_{+}}$, $a_j \geq 1$, is called a Bruno sequence if it is nondecreasing and
\begin{equation*}
	\sum_{j=1}^{\infty} 2^{-j} \ln a_j < \infty.
\end{equation*}
\end{definition}

\begin{definition} \label{Brunodef}
A vector $\lambda \in \mC^n$ satisfies the Bruno condition if the sequence $\{a_j\}_{j\in \mZ_{+}}$, $a_j = \max \{1, \Omega_{2^j + 1}\}$, is a Bruno sequence, where the numbers $\Omega_s$ are defined above.
\end{definition}

\begin{theorem} \label{Brunoth}
(Siegel–Bruno) Assume that

(1) $\lambda \in \mC^n$ satisfies the Bruno condition,

(2) $\widehat{u} \in \calA\cap \calF_{\diamd}$,

(3) the normal form of $ \Lambda z + \widehat{u}$ is $\Lambda z$.

Then there exists an analytic change of variables $\mathbf{w} \mapsto \bzz=\nu(\mathbf{w})$ that transforms the vector field $\Lambda z + \widehat{u}$ to its normal form.
\end{theorem}

We will prove this theorem in the sections below.

\subsection{Inductive step}
Here we prove an auxiliary lemma used to establish the Siegel–Bruno theorem in the following sections.

\begin{definition}
A sequence $\{b_s\}_{s=1}^{\infty}$ is called sublinear if $\lim_{j\to \infty} b_j/j = 0.$
\end{definition}

\begin{definition}
A sequence $\{b_s\}_{s=1}^{\infty}$ is called convex if $b_{j-1} - 2b_j + b_{j+1} \geq 0$ for every $j \in \mN$.
\end{definition}

\begin{lemma} \label{unductivelem}
Let the vector field $\Lambda z + \widehat{u}$, with $\widehat{u} \in \calA^{\rho} \cap \calF_{\diamd}$, satisfy
\begin{equation*}
\begin{split}
&(1)\; \widehat{u}^m = \sum_{|\bkk|\geq r} \widehat{U}_{\bkk}^m \bzz^{\bkk},\qquad 
|\widehat{U}_{\bkk}^m| \leq c\, e^{\,b_{|\bkk|}+ \alpha |\bkk|},\qquad \alpha \geq 0,\qquad r\geq 2, \\
&(2)\; \text{the vector of eigenvalues of }\Lambda \text{ satisfies the Bruno condition (see Def.~\ref{Brunodef});}\\
&(3)\; \{b_j\} \text{ is sublinear, convex, nonpositive, and nonincreasing;}\\
&(4)\; \text{the normal form of } \Lambda z + \widehat{u} \text{ is } \Lambda z.
\end{split}
\end{equation*}
Then there exists a change of variables $\bz \mapsto \bw= \nu(\bzz)$ transforming the vector field $\Lambda z + \widehat{u}$ into $\Lambda z + g$, where
\begin{equation*}
g^m = \sum_{|\bkk| \geq 2r} G_{\bkk}^m \mathbf{w}^{\bkk},\qquad 
|G_{\bkk}^m| \leq c\, e^{\,b_{\bkk} + (\alpha + \varepsilon) |\bkk|},
\end{equation*}
\begin{equation}\label{definitionofeps}
\varepsilon = c\, \Delta \,\Omega_{2r-2}, 
\qquad 
\Delta = (2 r)^n e^{\alpha} n \exp(2 b_r - b_{2r-1}).
\end{equation}
Moreover, for any $0<\rho\leq e^{-\alpha}$ and any $\rho'$ satisfying
\begin{equation}\label{estiomforrho}
0 < \rho' \leq \rho - \frac{1}{e^{\alpha} n}\, \varepsilon \,(e^{\alpha} \rho)^r,
\end{equation}
the following hold:

1. There is an analytic map $\nu\colon D_{\rho'} \to \nu (D_{\rho'}) \subseteq D_{\rho}$.

2. For the Jacobian matrix $D\nu(\bz)$, $\bz \in D_{\rho'}$, we have
\begin{equation*}
\begin{aligned}
\text{(i)}\quad & 
\exp\!\bigl(-\varepsilon'(\rho e^{\alpha})^{r-1}\bigr)
\leq \det D\nu(\bzz)
\leq \exp\!\bigl(\varepsilon'(\rho e^{\alpha})^{r-1}\bigr), \\[0.4em]
\text{(ii)}\quad & 
\| D\nu - I \|_{\rho'}
\leq \varepsilon'(\rho e^{\alpha})^{r-1} 
\exp\!\bigl(\varepsilon'(\rho e^{\alpha})^{r-1}\bigr), \\[0.4em]
\text{(iii)}\quad & 
\| D\nu \|_{\rho'}
\leq \exp\!\bigl(\varepsilon'(\rho e^{\alpha})^{r-1}\bigr).
\end{aligned}
\end{equation*}
where
\[
\Delta' = (2r)^{n+1} e^{\alpha} n 
\exp(2b_r - b_{2r-1}), 
\qquad 
\varepsilon' = c\,\Delta'\,\Omega_{2r-2}.
\]

3. The map $\nu$ is invertible, and $\nu^{-1}: \nu (D_{\rho'}) \to D_{\rho'}$ is analytic.

4. The normal form of the transformed vector field $\Lambda z + g$ is $\Lambda z$.
\end{lemma}

\begin{proof}
We construct a change of variables using continuous averaging with the operator $\xi_r$
\begin{equation*}
	(\xi_r u )^m = - \sum_{\bkk \in \calZ_r} e^{-i \arg{ \langle \lambda ,\bkk_m \rangle}} U_{\bkk} ^m \bzz^{\bkk},\qquad 
	\calZ_r = \{\bkk \in \mZ_{\diamd}^n \colon r \leq |\bkk| \leq 2r-2 \}.
\end{equation*}
Note that by assumption (4) of Lemma \ref{unductivelem} we have $\langle \lambda, \bkk_m \rangle \neq 0$. The flow $\psi^{\delta}$ associated with $\xi_r$ is given by
\begin{equation} \label{eq: flowforxir}
	\pa_{\delta}u = -[\xi_r u,\Lambda z + u],\qquad u \big|_{\delta=0} = \widehat{u}.
\end{equation}
Consider \eqref{eq: flowforxir} componentwise. Using assumption (1) of Lemma \ref{unductivelem}, for the components $U_{\mathbf{d}}^m$ with $|\mathbf{d}|<r$ we obtain
\begin{equation*}
	\pa_{\delta} U_{\mathbf{d}} ^m = 0, \qquad U_{\mathbf{d}} ^m \big|_{\delta=0} = 0, \qquad  m \in \{1,\ldots,n\}.
\end{equation*}
Hence the vector field $u = \psi^{\delta}(\widehat{u})$ can be written as $u = \Lambda z + u_r + u_{*}$, where
\begin{equation*}
	u_r ^m = \sum_{\bkk \in \calZ_r} U_{\bkk}^m (\delta) \bzz^{\bkk} ,\qquad 
	u_{*} ^m = \sum_{|\bkk|\geq 2r-1} U_{\bkk} ^m(\delta) \bzz^{\bkk}.
\end{equation*}

We solve \eqref{eq: flowforxir} componentwise. Collect all terms at $\bz^{\bdd} e_m$ with $\bdd \in \calZ_r$. Then \eqref{eq: flowforxir} becomes
\begin{equation} \label{detailedonsmallk}
	\pa_{\delta} U_{\bdd}^m = -|\langle \lambda, \bdd_m \rangle| U_{\bdd}^m - [\xi_{r} u, u_{r}]_{\bdd} ^m,
\end{equation}
where the second term is the sum of the terms at $\bzz^{\bdd} e_m$ of the form
\begin{equation*}
 [\xi_{r} u, u_{r}]_{\bdd} ^m 
 = \sum_{p=1}^n \Bigg( 
 \sum_{\substack{\bkk + \bss - \bdd = e_p,\\\bss,\bkk \in \calZ_r}} s_p U_{\bkk} ^p U_{\bss}^m 
 -\sum_{\substack{\bkk + \bss - \bdd = e_p,\\ \bss, \bkk \in \calZ_r}} s_p  U_{\bss} ^m U_{\bkk}^p
 \Bigg).
\end{equation*}

In \eqref{detailedonsmallk} the summation variables $\bss,\bkk \in \calZ_r$ must satisfy
\begin{equation}\label{aineqtrivial}
	r+1 \leq|\bkk| + |\bss| \leq 2r -1.
\end{equation}
But \eqref{aineqtrivial} cannot hold when $\bss, \bkk \in \calZ_r$. Therefore \eqref{detailedonsmallk} is equivalent to
\begin{equation}\label{ksmalldetailedform}
	\pa_{\delta} U_{\bkk}^m = -|\langle \lambda, \bkk_m \rangle| U_{\bkk}^m, \qquad \bkk \in \calZ_r.
\end{equation}
Equation \eqref{ksmalldetailedform} is easily solved:
\begin{equation} \label{solutionofZr}
	U_{\bkk}^m(\delta) = e^{-|\langle \lambda, \bkk_m \rangle| \delta} \widehat{U}_{\bkk} ^m, \qquad \bkk \in \calZ_r.
\end{equation} 

Now consider \eqref{eq: flowforxir} and collect all terms at $\bz^{\bdd} e_m$ with $|\bdd|\geq 2r-1$. Substituting \eqref{solutionofZr} into \eqref{eq: flowforxir} yields
\begin{equation}\label{Equation4inindlem}
\begin{split}
	\pa_{\delta}U_{\bdd}^m &= \sum_{p=1}^n \Bigg( 
	\sum_{\substack{\bkk + \bss - \bdd = e_p,\\\bkk \in \calZ_r}} s_p e^{-|\langle \lambda, \bkk_p \rangle| \delta} \widehat{U}_{\bkk} ^p U_{\bss}^m 
	-\sum_{\substack{\bkk + \bss - \bdd = e_p,\\ \bss \in \calZ_r}} s_p e^{-|\langle \lambda, \bss_m \rangle| \delta} \widehat{U}_{\bss} ^m U_{\bkk}^p\Bigg),\\ 
	U_{\bkk}^m (0) &= \widehat{U}_{\bkk}^m.
\end{split}
\end{equation}

Associate to \eqref{Equation4inindlem} the majorant functions
\begin{equation} \label{connectedmaj}
	\bU = \sum_{\bkk \in \mZ_{\diamd}^n} \bU_{\bkk} \bzz^{\bkk},\quad |U_{\bkk}^m| \leq \bU_{\bkk},\qquad 
	\widehat{\bU} = \sum_{\bkk\in \mZ_{\diamd}^n} \widehat{\bU}_{\bkk} \bzz^{\bkk},\quad 
	|\widehat{U}_{\bkk}^m| \leq \widehat{\bU}_{\bkk} = c\, e^{\,b_{|\bkk|} + \alpha |\bkk|}.
\end{equation}
In \eqref{connectedmaj} we may take $\bU_{\bkk} = \max_{1\leq m\leq n} |U_{\bkk}^m|$. Using the Bruno condition \eqref{Brunodef}, the majorants \eqref{connectedmaj}, and replacing minus signs by plus in \eqref{Equation4inindlem}, we obtain the majorant system
\begin{equation} \label{majoreq}
	\pa_{\delta} \bU_{\bdd} = 2 e^{-\delta/\Omega_{2r-2}} |\bdd| 
	\sum_{p=1}^n \sum_{\substack{\bkk + \bss - \bdd = e_p,\\\bkk \in \calZ_r}} 
	\widehat{\bU}_{\bkk}\,\bU_{\bss},\qquad 
	\bU_{\bdd} (0) = \widehat{\bU}_{\bdd},
\end{equation}
where $\Omega_{2r-2}$ is defined in \eqref{def: Omegasdef}.

Substituting $\bU_{\bkk} = c\, e^{\,b_{|\bkk|} + \alpha |\bkk|}\, \bu_{\bkk}$ into \eqref{majoreq}, we get
\begin{equation}\label{majdifeq}
\begin{split}
\pa_{\delta} \bu_{\bdd} &= e^{-\delta/\Omega_{2r-2} } \Sigma(\bdd), \qquad \bu_{\bdd} (0) = 1, \ \text{where}\\
\Sigma(\bdd) &:= 2 |\bdd| c  \sum_{p=1}^n \sum_{\substack{\bkk + \bss - \bdd = e_p,\\\bkk \in \calZ_r}} 
\exp\!\big(b_{|\bkk|} + b_{|\bss|} - b_{|\bdd|}\big)\, e^{\alpha}\, \bu_{\bss}.
\end{split}
\end{equation}

We use the notation
\begin{equation*}
	\chi_K(\delta) = \max_{|\bdd| \leq K} \bu_{\bdd} (\delta),\qquad \delta \geq 0.
\end{equation*}

\begin{lemma} \label{auxlem}
Let the sequence $\{b_j\}_{j=0}^n$ be convex and nonincreasing. Then for any $\bdd$ with $|\bdd|\ge 2r-1$ the following inequality holds:
\begin{equation*}
	\Sigma (\bdd) \leq c\, |\bdd|\,\Delta \, \chi_{|\bdd|}, \quad \text{where $\Delta$ is defined in \eqref{definitionofeps}.}
\end{equation*}
\end{lemma}

We will prove Lemma \ref{auxlem} at the end of this section. Consider the following equation:
\begin{equation}\label{majeq1}
	\pa_{\delta} \chi_{|\bdd|} = c\, |\bdd|\,\Delta \, e^{-\delta/\Omega_{2r-2} } \chi_{|\bdd|}, 
	\qquad \chi_{|\bdd|} (0) = 1.
\end{equation}

The solution of \eqref{majeq1} is
\begin{equation*}
	\chi_{|\bdd|}(\delta) = \exp\Big(c \Delta |\bdd| \int_{0}^{\delta} e^{- \widetilde{\delta}/\Omega_{2r-2}} \, d \widetilde{\delta}\Big).
\end{equation*}

Using Lemma \ref{auxlem}, we obtain
\begin{equation} \label{eq: auxestim}
	\pa_{\delta} \big(\chi_{|\bdd|} - \bu_{\bdd}\big) 
	=  e^{- \delta/\Omega_{2r-2}}\left(c\, |\bdd|\,\Delta \, \chi_{|\bdd|} -\Sigma(\bdd)  \right) \geq 0, 
	\qquad  \chi_{|\bdd|}(0) - \bu_{\bdd} (0) =0.
\end{equation}

From \eqref{eq: auxestim} it follows that
\begin{equation*}
	0\leq \bu_{\bdd} (\delta) \leq \chi_{|\bdd|} (\delta) 
	= \exp\Big(c \Delta |\bdd| \int_{0}^{\delta} e^{- \widetilde{\delta}/\Omega_{2r-2}} \, d \widetilde{\delta}\Big).
\end{equation*}

At $\delta = + \infty$ we get
\begin{equation*}
	\bu_{\bdd} (+\infty) \leq \exp\!\big(c \Delta |\bdd| \Omega_{2r-2}\big).
\end{equation*}

Thus, for the transformed vector field $\Lambda z + g$, the Taylor coefficients have the following estimate: 
\begin{equation*}
	|G_{\bkk}^m| \leq \bU_{\bdd} (+\infty) \leq c\, \exp \left( b_{|\bdd|} + |\bdd|(\alpha + \varepsilon) \right).
\end{equation*}

Note that the linear part of the transformed vector field $\Lambda z + g$ remains unchanged. Therefore the set $\mathbf{L}_{\lambda}$ does not change (and remains trivial) during the transformation of the initial vector field. Hence item 4 of the lemma holds.

Let $g_{0,\delta}$ be the $\delta$–shift along solutions of the system:
\begin{equation} \label{eq: changeofvar}
\begin{split}
	\dot{z_j} (\delta) &= (\xi_r u)^j,\quad 
	(\xi_r u)^j = - \sum_{\bkk \in \calZ_r} e^{-i \arg\langle \lambda, \bkk_j\rangle} e^{-|\langle \lambda, \bkk_j \rangle |\delta} \widehat{U}_{\bkk} ^j \bzz^{\bkk},\\
	z_j(0)&=z_j,\quad j=1,\ldots,n.
\end{split}
\end{equation}

We estimate $\rho'$ so that for each $\bzz=\bzz(0) \in D_{\rho'}$ and for all $\delta \in [0,+\infty]$ the point $\bzz(\delta) = g _{\delta} (\bzz)$ does not leave $D_{\rho}$. As long as $\bzz(\delta) = (z_1(\delta),\ldots,z_n(\delta)) \in D_{\rho}$, we estimate
\begin{equation} \label{eq: estimfordotz}
\begin{split}
	|\dot{z}_j (\delta)| &\leq \sum_{\bkk \in \calZ_r} |U_{\bkk}^j (\delta)| \rho^{|\bkk|}
	= \sum_{\bkk \in \calZ_r} e^{-|\langle \lambda, \bkk_j \rangle|\delta}|\widehat{U}_{\bkk}^j| \rho^{|\bkk|}\\
	&\leq c\, e^{-\delta / \Omega_{2r-2}} \sum_{\bkk \in \calZ_r} e^{b_{\bkk} + \alpha |\bkk|} \rho^{|\bkk|}
	\leq c\, e^{-\delta / \Omega_{2r-2}} e^{b_r} \sum_{\bkk \in \calZ_r} (e^{\alpha} \rho)^{|\bkk|}\\
	&\leq c\, e^{-\delta / \Omega_{2r-2}} e^{b_r} (2r)^n (\rho e^{\alpha})^{r}.
\end{split}
\end{equation}

Here we used that the sequence $\{b_s\}$ is nonincreasing, $|\bkk| \leq 2r$, that $\rho e^{\alpha} \leq 1$, and that the sum has fewer than $(2r)^n$ terms. As a result we obtain
\begin{equation*}
	|z_j(\delta) - z_j(0)| \leq c\, \Omega_{2r-2} (2r)^n e^{b_r} (\rho e^{\alpha})^r. 
\end{equation*}

Since the sequence $\{b_s\}_{s=1}^{\infty}$ is nonincreasing, we have $b_r - b_{2r-2} \geq 0$. Hence
\begin{equation} \label{eq: inequa}
	|z_j(\delta) - z_j(0)| \leq c\,\Omega_{2r-2}  (2r)^n e^{2b_r - b_{2r-1}} (\rho e^{\alpha})^r 
	= \frac{1}{e^{\alpha} n} \varepsilon (\rho e^{\alpha})^r.
\end{equation}

If
\begin{equation*}
	|z_j(0)|=|z_j|<\rho'\leq \rho -\frac{1}{e^{\alpha} n} \varepsilon (\rho e^{\alpha})^r, 
\end{equation*}
then from \eqref{eq: inequa} it follows that for every $\delta\geq 0$,
\begin{equation} \label{ineq: forzjdelta}
	|z_j (\delta)| \leq |z_j(0)| + \frac{1}{e^{\alpha} n} \varepsilon (\rho e^{\alpha})^r < \rho.
\end{equation}

Thus we have a $\delta$–shift that never leaves the polydisk $D_{\rho}$:
\begin{equation*}
	g_{0,\delta}\colon D_{\rho'} \to g_{0,\delta} (D_{\rho'}) \subseteq D_{\rho},\qquad 
	g_{0,\delta} (\bzz) = \bzz_\delta, \qquad \delta \geq 0.
\end{equation*}

Since the right-hand side of \eqref{eq: changeofvar} is a polynomial of degree $2r-2$ in $\bzz$ depending continuously on $\delta$, the $\delta$–shift $g_{\delta}$ is analytic in $D_{\rho'}$.

Take the solution $\bzz_{\delta}=g_{0,\delta}(\bzz)$ of \eqref{eq: changeofvar}. We show that the Jacobian matrix of the analytic function $\bzz_{\delta}=g_{0,\delta}(\bzz)$ is nondegenerate at $\bzz \in D_{\rho'}$. The Jacobian with respect to the initial condition $\bzz \in D_{\rho'}$ is given by
\begin{equation*}
	X(\delta) = \pa_{\bzz} \left( g_{0,\delta}(\bzz)\right), \qquad X(0) = E, 
\end{equation*}
where $E$ is the identity matrix. Differentiating \eqref{eq: changeofvar} with respect to $\bzz$ we obtain the variational system (see §32 in \cite{Arnold2}):
\begin{equation} \label{eq: varsystem}
	\dot{X} = -\pa_{\bzz} \left( \xi_r u (g_{0,\delta}(\bzz)) \right) X, \qquad X(0)=E.
\end{equation}

By Liouville’s formula,
\begin{equation} \label{eq: liuville}
	\det X(\delta) = \exp \left( -\int_{0}^{\delta}  \text{tr}\, \pa_{\bzz} \left( \xi_r u (g_{0,\delta}(\bzz)) \right) \right) \neq 0,\qquad \delta \geq 0.
\end{equation}

We bound the modulus of the diagonal entries of the matrix $\pa_{\bzz} \left( \xi_r u (g_{0,\delta}(\bzz)) \right)$. We will estimate the entries of the matrix $\pa_{\bzz} \left( \xi_r u (g_{0,\delta}(\bzz)) \right)$ using the estimates in \eqref{eq: estimfordotz}:
\begin{equation} \label{ineq: ineq1}
\begin{split}
\left|	\frac{\pa (\xi_r u)^i}{\pa z_j} \right| 
&\leq  c\, e^{-\delta / \Omega_{2r-2}} e^{2b_r - b_{2r-1}} (2r)^{n+1} e^{\alpha} (\rho e^{\alpha})^{r-1}\\
&\leq \frac{\varepsilon' (\rho e^{\alpha})^{r-1}}{n} \frac{1}{\Omega_{2r-2}} e^{-\delta / \Omega_{2r-2}} .
\end{split}
\end{equation} 

Using \eqref{eq: liuville} we derive
\begin{equation*}
\begin{split}
	\det X(\delta) &\leq \exp \left(\varepsilon' (\rho e^{\alpha})^{r-1}\int_{0} ^{\delta} \frac{1}{\Omega_{2r-2}} e^{-\tilde{\delta} / \Omega_{2r-2}} \, d \tilde{\delta}\right),\\
	\det X(\delta)&\geq \exp \left(- \varepsilon' (\rho e^{\alpha})^{r-1}\int_{0} ^{\delta} \frac{1}{\Omega_{2r-2}} e^{-\tilde{\delta} / \Omega_{2r-2}} \, d \tilde{\delta}\right)>0, \qquad \delta \geq 0,
\end{split}
\end{equation*}
where $\Delta'$ and $\varepsilon'$ are as in the statement of the lemma. It follows that, as $\delta \to +\infty$,
\begin{equation*}
	\exp \!\big(-\varepsilon' (\rho e^{\alpha})^{r-1}\big) 
	\leq \det X(+ \infty) 
	\leq \exp \!\big(\varepsilon' (\rho e^{\alpha})^{r-1}\big), 
	\quad X(+\infty) =\nu(\bzz)= D g_{0,+\infty} (\bzz).
\end{equation*} 

This proves (i). To prove (iii), using \eqref{eq: varsystem} and \eqref{ineq: ineq1} we get
\begin{equation*}
\begin{split}
	\| X(\delta) \|_{\rho'} 
	&\leq 1 + \int_{0}^{\delta} \big\|\pa_{\bzz} \left( \xi_r u (g_{0,s}(\bzz)) \right)\big\|_{\rho'} \|X(s) \|_{\rho'} \, ds\\
	& \leq 1 + \varepsilon' (\rho e^{\alpha})^{r-1} \int_{0}^{\delta} \frac{1}{\Omega_{2r-2}} e^{-s / \Omega_{2r-2}} \|X(s) \|_{\rho'} \, ds .
\end{split}
\end{equation*}

By Grönwall’s lemma,
\begin{equation*}
	\| X(\delta) \|_{\rho'} \leq \exp \left( \varepsilon' (\rho e^{\alpha})^{r-1}\right).
\end{equation*} 
Letting $\delta \to +\infty$ yields (iii). To prove (ii), from \eqref{eq: varsystem} we have
\begin{equation*}
	X(\delta) - I 
	= -\int_{0}^{\delta}\pa_{\bzz} \left( \xi_r u (g_{0,s}(\bzz)) \right) \, ds  
	-\int_{0}^{\delta}\pa_{\bzz} \left( \xi_r u (g_{0,s}(\bzz)) \right) \big(X(s) - I\big)\, ds.
\end{equation*}
Hence
\begin{equation*}
\| X(\delta) - I \|_{\rho'} 
\leq \int_{0}^{\delta} \big\|\pa_{\bzz} \left( \xi_r u (g_{0,s}(\bzz)) \right) \big\|_{\rho'}  \, ds
+\int_{0}^{\delta} \big\|\pa_{\bzz} \left( \xi_r u (g_{0,s}(\bzz)) \right) \big\|_{\rho'} \|X(s) \|_{\rho'} \, ds.
\end{equation*}
Applying Grönwall’s lemma gives (ii). In conclusion:

1) the shift $\bzz_{\delta} = g_{0,\delta} (\bzz)$ is analytic in $D_{\rho'}$;

2) the Jacobian matrix of $g_{0,\delta}(\bzz)$ is nondegenerate at every $\bzz \in D_{\rho'}$ for all $\delta \in [0,+\infty]$.

By the inverse function theorem, the inverse map (the backward shift) $g^{-1}_{0,\delta} = g_{\delta,0}\colon g_{0,\delta}(D_{\rho'}) \to D_{\rho'}$ is analytic on $g_{0,\delta}(D_{\rho'}) \subseteq D_{\rho}$. Set $\nu = g_{+\infty,0}$.
\end{proof}

\begin{proof}[of Lemma \ref{auxlem}]
We show that the following inequality holds:
\begin{equation} \label{ineqforb}
	b_{|\bkk|}+ b_{|\bss|} - b_{|\bdd|} \leq 2 b_r - b_{2r-1}.
\end{equation}

Using the second inequality from Lemma \ref{axlemmas} with $m = |\bkk| - r$, $k = |\bkk|$, $l = |\bss|$, we obtain
\begin{equation}\label{ineqaux1}
	b_{|\bkk|} + b_{|\bss|} \leq b_r + b_{|\bdd|+1+r}, 
	\quad |\bdd| + 1 = |\bkk| + |\bss|, 
	\quad |\bdd|> 2r-2.
\end{equation}

Since the sequence $\{b_n\}_{n=1}^{\infty}$ is convex and nonincreasing, we have
\begin{equation}\label{ineqaux2}
	b_{|\bdd| + 1 - r} - b_{|\bdd|} \leq b_r - b_{2r-1}.
\end{equation}

Adding \eqref{ineqaux1} and \eqref{ineqaux2} gives \eqref{ineqforb}. Using \eqref{ineqforb}, we can estimate $\Sigma(\bdd)$ from \eqref{majdifeq}:
\begin{equation}\label{sigmaineq1}
	\Sigma(\bdd) \leq 2 |\bdd|\, c \,  \exp(2b_r - b_{2r-1}) e^{\alpha}  \chi_{|\bdd|}
	\sum_{p=1}^n\sum_{\substack{\bkk + \bss - \bdd = e_p,\\\bkk \in \calZ_r}} 1.
\end{equation}

Moreover, we have
\begin{equation}\label{ineqforamountsofk}
	\# \{\bkk' \in \mZ_{\diamd}^n \colon |\bkk'|=|\bkk|,\quad k \in \calZ_r\} <(2r)^{n-1}.
\end{equation}

Applying \eqref{ineqforamountsofk} to \eqref{sigmaineq1}, we obtain
\begin{equation*}
	\Sigma(\bdd) \leq 2 |\bdd|\, c \,  \exp(2b_r - b_{2r-1}) e^{\alpha}  \chi_{|\bdd|}\, n (2r)^{n-1} 
	\leq c\,|\bdd|\, \Delta \, \chi_{|\bdd|},
\end{equation*}
where $\Delta$ is defined in \eqref{definitionofeps}. This proves Lemma \ref{auxlem}.
\end{proof}

\subsection{Proof of the Siegel–Bruno Theorem}

In this section we prove the Siegel–Bruno theorem. Before giving the proof, we state an auxiliary fact. Consider the vector field $\Lambda z + \widehat{u}$. Let $\lambda \in \mC^n$ be the vector of eigenvalues of the matrix $\Lambda$. Using Lemma \ref{majornatlem}, we have
\begin{equation*}
	\widehat{u}^m = \sum_{|\bkk| \geq 2} \widehat{U}_{\bkk} ^m \bzz^{\bkk},\qquad 
	|\widehat{U}_{\bkk}^m| \leq \widehat{c}\, e^{\,b_{|\bkk|} + \widehat{\alpha}_0 |\bkk|},\qquad m=1,\ldots,n.
\end{equation*}

Choosing $\alpha_0 \geq \widehat{\alpha}_0$, we obtain
\begin{equation}\label{estimforUk}
	|\widehat{U}_{\bkk} ^m| \leq c\, e^{\,b_{|\bkk|}+\alpha_0 |\bkk|}, \qquad |\bkk|\geq 2,
\end{equation}
where $c = \widehat{c}\,e^{\,2(\widehat{\alpha}_0 - \alpha_0)}$. Hence, by taking $\alpha_0$ sufficiently large, we may assume
\begin{equation}\label{someinequalities}
	c\, e^{\alpha_0} \leq \tfrac{1}{8},\qquad n\, e^{\alpha_0} \geq 2.
\end{equation}

\begin{theorem} \label{commonofsiegelth}
Assume that
\begin{itemize}
	\item[(1)] $\lambda \in \mC^n$ satisfies the Bruno condition;
	\item[(2)] $\widehat{u} \in \calA^{\rho_0}\cap \calF_{\diamd}$ with $\rho_0 = e^{-\alpha_0}$, where $\alpha_0$ is as in \eqref{someinequalities};
	\item[(3)] the normal form of $ \Lambda z + \widehat{u}$ equals $\Lambda z + O_{r-1}(z)$.
\end{itemize}
Then there exists an analytic change of variables
\begin{equation}\label{thelimitofrho}
	\nu\colon D_{\rho_0} \to D_{\rho_*},\qquad \mathbf{w} \mapsto \bzz = \nu(\mathbf{w}),\qquad \rho_{*} = \rho_0 e^{-1/2},
\end{equation}
which transforms the vector field $\Lambda z + \widehat{u}$ into a vector field $\Lambda z + g$, where
\begin{equation}\label{gfunctioninth}
	g^m(\mathbf{w}) = \sum_{|\bkk| \geq r} G_{\bkk} ^m \mathbf{w}^{\bkk} = O_r(\mathbf{w}),\qquad 
	|G_{\bkk}^m| \leq c\, e^{\,b_{|\bkk|} + \beta |\bkk|},\qquad \beta \leq \alpha_0 + \tfrac{1}{2}.
\end{equation}
\end{theorem}

\begin{proof}[of Theorem \ref{commonofsiegelth}]
We apply Lemma \ref{unductivelem} inductively. Let
\[
	N = \max\{m \in \mN \colon 2^{m-1}+1 < r\}.
\]
Then the value of $r$ in Lemma \ref{unductivelem} runs through
\[
	r_1, r_2, \ldots, r_N,\qquad r_m = 2^{m-1}+1.
\]

Thus we obtain sequences $\alpha_0,\alpha_1,\ldots,\alpha_N$, $\varepsilon_1,\ldots,\varepsilon_N$, $\Delta_1,\ldots,\Delta_N$, $\varepsilon'_1,\ldots,\varepsilon'_N$, $\Delta'_1,\ldots,\Delta'_N$ which control the system at the $m$-th step:
\begin{equation} \label{sequences}
\begin{split}
	\alpha_{m+1} &= \alpha_m + \varepsilon_{m+1},\\
	\varepsilon_m &= c\,\Delta_m\,\Omega_{r_m}, \quad 
	    \Delta_m = (2r_m)^n e^{\alpha_{m-1}} n\, \exp\!\big(2b_{r_m} - b_{r_{m+1}}\big),\\
	\varepsilon'_m &= c\,\Delta'_m\,\Omega_{r_m},\quad 
	    \Delta'_m = (2r_m)^{n+1} e^{\alpha_{m-1}} n\, \exp\!\big(2 b_{r_m} - b_{r_{m+1}}\big).
\end{split}
\end{equation}

Set
\begin{equation}\label{seqaj}
	a_m = \exp\!\big( b_{r_{m+1}} - 2 b_{r_m}\big).
\end{equation}
Using Lemma \ref{aboutexistaj}, we obtain the existence of a sequence $b_s$ satisfying
\begin{equation}\label{eqexistb}
	n\,2^{m} (2 r_m)^{n+1} \Omega_{r_m} 
	= \exp\!\big( b_{r_{m+1}} - 2 b_{r_m}\big).
\end{equation}

From \eqref{eqexistb} we get
\[
	\varepsilon_m 
	= c (2r_m)^n e^{\alpha_{m-1}} n \exp\!\big(2b_{r_m} - b_{r_{m+1}}\big) \Omega_{r_m} 
	= \frac{c\, e^{\alpha_{m-1}}}{2^{m+1} r_m}
	= \frac{c\, e^{\alpha_0 + \varepsilon_1 + \cdots + \varepsilon_{m-1}}}{2^{m+1} r_m}.
\]

We prove by induction that, if $c e^{\alpha_0} \leq 1/8$, then $\varepsilon_m \leq 2^{-m-2}$. Indeed, for $m=1$ we have
$\varepsilon_1 < \frac{c e^{\alpha_0}}{2} \leq \frac{1}{16} < 2^{-3}.$
Assume the bound holds for $m \leq m_0$. Then
\[
	\varepsilon_m \leq \frac{1}{8}\, e^{2^{-3} + \cdots + 2^{-m-1}} \frac{1}{2^m}
	\leq \frac{e^{1/2}}{8}\,\frac{1}{2^m} < 2^{-m-2}.
\]

Using \eqref{eqexistb} we also get
\begin{equation} \label{eq: estimforvareps}
	\varepsilon'_m 
	= c (2r_m)^{n+1} e^{\alpha_{m-1}} n \exp\!\big(2b_{r_m} - b_{r_{m+1}}\big) \Omega_{r_m} 
	= \frac{c e^{\alpha_{m-1}}}{2^{m}} 
	\leq \varepsilon_m \leq \frac{1}{2^{m+2}}.
\end{equation}

Setting $\beta = \alpha_N$, we obtain $\beta - \alpha_0 = \sum_{j=1}^{N} \varepsilon_j < 1/2$, which yields \eqref{gfunctioninth}.

To estimate $\rho_{*}$, define a sequence of maps $\nu_1,\ldots,\nu_N$, where $\nu_m$ is the map from Lemma \ref{unductivelem} at step $m$. Define
\[
	\rho_0,\ldots,\rho_N,\qquad \rho_0 = e^{-\alpha_0}, \quad \rho_N = \rho_{*},\quad 
	\rho_{m+1} = \rho_m e^{-\varepsilon_{m+1}}.
\]
Then $\rho_m e^{\alpha_m} 
	= \rho_{m-1} e^{-\varepsilon_m+ \alpha_m}
	= \rho_{m-1} e^{\alpha_{m-1}}$,
hence $\rho_m e^{\alpha_m} = \rho_0 e^{\alpha_0} = 1$ and $\rho_{\star} = \rho_N = \rho_0 e^{-\sum_{j=1}^{N} \varepsilon_j} \ge \rho_0 e^{-1/2}.$
At the $m$-th step, the inequality \eqref{estiomforrho},
\[
	\rho_{m+1} \le \rho_m - \frac{1}{e^{\alpha_m} n}\, \varepsilon_{m+1},
\]
holds if and only if
\[
	e^{-\varepsilon_{m+1}} \le 1 - \frac{\varepsilon_{m+1}}{n},
	\qquad \varepsilon_{m+1} \le 2^{-m-3}.
\]

By Lemma \ref{unductivelem}, we have $\nu_m \colon D_{\rho_m} \to \nu_m(D_{\rho_m}) \subseteq D_{\rho_{m-1}}$, and for each $\nu_m$:
\begin{enumerate}
\item $\nu_m$ is analytic and injective on $D_{\rho_m}$;
\item the Jacobian matrix of $\nu_m$ is nondegenerate at every $\bzz \in D_{\rho_m}$, with
\begin{equation} \label{eq: ineqfordet}
	\exp(-\varepsilon'_m) \le \det D\nu_m(\bzz) \le \exp(\varepsilon'_m);
\end{equation}
\item $\nu_m(D_{\rho_m}) \subseteq D_{\rho_{m-1}}$.
\end{enumerate}

Therefore, the mapping
\begin{equation} \label{eq: compostionofnu}
	F_N = \nu_1 \circ \cdots \circ \nu_N \colon D_{\rho_N} \to F_N(D_{\rho_N}) \subseteq D_{\rho_0}
\end{equation}
is well-defined. Estimating the Jacobian determinant of $F_N$ and using \eqref{eq: ineqfordet} and \eqref{eq: estimforvareps}, we obtain
\begin{equation} \label{ineq: ineqfordetD}
\begin{split}
	\det D F_N (\bzz) &\le \exp\!\left(\sum_{m=1}^N \varepsilon'_m\right) 
	\le \exp\!\left(\sum_{m=1}^{N} 2^{-m-2}\right) < e^{1/4},\\
	\det D F_N (\bzz) &\ge \exp\!\left(-\sum_{m=1}^{N} 2^{-m-2}\right) > e^{-1/4}.
\end{split}
\end{equation}

Since $D_{\rho_*} \subseteq D_{\rho_N}$, we can define the map $F_N$ on the polydisk $D_{\rho_*}$. By the nondegeneracy of Jacobian matrix of $F_N$, its inverse
$F_N^{-1}\colon F_N(D_{\rho_*}) \to D_{\rho_*}$ is analytic. Set $\nu = F_N^{-1}$.
\end{proof}

\textit{Proof of Theorem \ref{Brunoth}}.
We will show that the sequence of analytic maps $\{F_N\}_{N=1}^{+\infty}$ is bounded by a common constant on $D_{\rho_*}$. For each map $\nu_m$ the following inequality holds:
\begin{equation} \label{ineq: numinequality}
	\|\nu_m\|_{\rho_m} \le \varepsilon_m + \exp(\varepsilon_m).
\end{equation}

To prove \eqref{ineq: numinequality} we use the Newton–Leibniz formula and the mean value theorem. Set $\gamma(t)=t\bz$, $\bz\in D_{\rho_m}$. Then
\begin{equation*}
	\nu_m(\bz) - \nu_m(0) = \int_0^1 D\nu_m(\gamma(t))\,\bz\, dt.
\end{equation*}

Using \eqref{ineq: forzjdelta}, \eqref{eq: estimforvareps}, and estimate (iii) from Lemma \ref{unductivelem}, we obtain
\begin{equation*}
	\|\nu_m\|_{\rho_m} \le \frac{\rho_m \varepsilon_m}{n} + \rho_{m-1}\exp(\varepsilon'_m)
	\le \varepsilon_m + \exp(\varepsilon_m).
\end{equation*}

Inequality \eqref{ineq: numinequality} yields
\begin{equation} \label{ineq: lipineq}
	\|F_N\|_{\rho_*} \le \prod_{j=1}^{N}\big(\varepsilon_j + \exp(\varepsilon_j)\big)
	\le \prod_{j=1}^{+\infty}\big(\varepsilon_j + \exp(\varepsilon_j)\big)=:L>0.
\end{equation}

The infinite product in \eqref{ineq: lipineq} converges. Hence the sequence $\{F_N\}$ is uniformly bounded by $L>0$ on the polydisc $D_{\rho_*}$. By Montel’s theorem (see \cite{DomSer}, p. 188), from $\{F_N\}$ we can extract a subsequence converging to an analytic map $F$. Next we show that the full sequence $\{F_N\}_{N=1}^{+\infty}$ is Cauchy with respect to $\|\cdot\|_{\rho_*}$. Let $M>N$. Taking into account the norm of the difference in the mappings $F_N$, $F_M$, using \eqref{ineq: lipineq}:
\begin{equation*}
\begin{split}
	\|F_M - F_N\|_{\rho_*}
	&= \big\|F_N \circ \big(\nu_{N+1}\circ\cdots\circ \nu_M - I\big)\big\|_{\rho_*}\\
	&\le L \,\big\|\nu_{N+1}\circ\cdots\circ \nu_M - I \big\|_{\rho_*}.
\end{split}
\end{equation*}

\begin{lemma} \label{lemaforcomp}
Let $M>N$, where $M,N\in\mN$. Consider the sequence of maps $\nu_s$ defined in Lemma \ref{unductivelem} on step $s$. Then
\begin{equation*}
	\big\|\nu_{N+1}\circ\cdots\circ \nu_M - I \big\|_{\rho_*}
	\le \sum_{j=M}^{N} \varepsilon_j \exp \varepsilon_j 
	     + \sum_{j=M}^{N} \|\nu_j(0)\|_{\infty},
\end{equation*}
where $\|z\|_{\infty}=\max_{j=1,\ldots,n}|z_j|$, $z\in\mC^n$, the sequence $\varepsilon_s$ is defined in \eqref{sequences}, and $\rho_*$ is as in Theorem \ref{commonofsiegelth}.
\end{lemma}

We prove this lemma below. Using it we obtain
\begin{equation*}
	\|F_M - F_N\|_{\rho_*}
	\le L\left(\sum_{j=M}^{N} \varepsilon_j \exp \varepsilon_j
	+ \sum_{j=M}^{N} \|\nu_j(0)\|_{\infty}\right).
\end{equation*}

By \eqref{ineq: forzjdelta} and \eqref{eq: estimforvareps},
\begin{equation*}
	\sum_{j=1}^{+\infty} \|\nu_j(0)\|_{\infty}
	\le \sum_{j=1}^{+\infty} \frac{1}{n}\rho_j \varepsilon_j
	\le \sum_{j=1}^{+\infty} \frac{1}{2^{j+2}} < +\infty,
	\qquad
	\sum_{j=1}^{+\infty} \varepsilon_j \exp \varepsilon_j <+\infty.
\end{equation*}

Therefore $\{F_N\}$ is Cauchy, hence $F=\lim_{N\to\infty} F_N$ is analytic. Passing to the limit in \eqref{ineq: ineqfordetD}, we obtain the nondegeneracy of the Jacobian of $F$:
\begin{equation*}
	\exp \left(-1/4\right) \le \det DF(\bz) \le \exp (1/4), \qquad \bz\in D_{\rho_*}.
\end{equation*}

We obtain the desired analytic map $\nu := F^{-1}$.

\begin{proof}[of Lemma \ref{lemaforcomp}]
Fix $N \in \mN$. Proceed by induction on $M$. Consider the case $M=1$. Consider
\begin{equation} \label{eq: numminusz}
	\nu_m (\bzz) - \bzz = h_m (\bzz) + \nu_m (0), \qquad h_m (\bzz) = \nu_m (\bzz) - \nu_m (0) - \bzz,\qquad \bzz \in D_{\rho_*}.
\end{equation}	

We will use the Mean Value Theorem and the Newton–Leibniz formula, with $\gamma(t) = \bzz t$, $t \in [0,1]$:
\begin{equation*}
	h_m (\bzz) = \int_{0} ^1 D h_m (\gamma(t)) \bzz \, ds, \qquad D h_m (\bw) = D \nu_m (\bw) - \bw.
\end{equation*}

Using inequality (ii) of Lemma \ref{unductivelem} and inequality \eqref{eq: estimforvareps}, we obtain
\begin{equation*}
	\| h_m\|_{\rho_*} \leq \varepsilon'_m \exp \varepsilon'_m \leq \varepsilon_m \exp \varepsilon_m.
\end{equation*}

Then, taking into account \eqref{eq: numminusz}, we get:
\begin{equation} \label{eq: numinI}
	\| \nu_m - I \|_{\rho_*} \leq \| \nu_m (0)\|_{\rho_*} + \varepsilon_m \exp (\varepsilon_m).
\end{equation}

Substituting $m=N+1$ into \eqref{eq: numinI} gives the base case of the induction. Suppose the lemma holds for all $M<K$. Let $M=K$; using the induction hypothesis and \eqref{eq: numinI}, we obtain:
\begin{equation*}
	\begin{split}
			\|\nu_{N+1} \circ \ldots \circ \nu_K - I \|_{\rho_*} &= \|\nu_{N+1} \circ \ldots \circ \nu_{K-1}\circ \nu_K - I \|_{\rho_*}\\
			&=\|\nu_{N+1} \circ \ldots \circ \nu_{K-1}\circ \nu_K - \nu_K + \nu_K- I \|_{\rho_*}\\
			&\leq \|\nu_{N+1} \circ \ldots \circ \nu_{K-1}\circ \nu_K - \nu_K\|_{\rho_*} + \| \nu_K- I \|_{\rho_*}\\
			& \leq \sum_{j=N}^{K-1} \varepsilon_j \exp \varepsilon_j +  \sum_{j=N}^{K-1} \| \nu_j (0)\|_{\infty} + \varepsilon_K \exp \varepsilon_K + \|\nu_K(0)\|_{\infty}.
	\end{split}
\end{equation*}
\end{proof}

\section{Technical part}
\subsection{Majorants}

For any $F,G \in \calF$, we say $F\ll G$ if and only if for the corresponding Taylor coefficients we have the inequalities
$|F_{\bkk} ^m| \leq G_{\bkk}^m,\quad \bkk \in \mZ^n_{\diamd},\quad m\in \{1,\ldots,n\}.$ 
\begin{lemma}
	Let $F \ll G$, $\widehat{F} \ll \widehat{G}$. Then
	\begin{equation*}
		\begin{split}
			&\mathbf{1}.\; F^m + \widehat{F}^s \ll G^m+\widehat{G}^s,\quad F^m \widehat{F}^s \ll G^m \widehat{G}^s,\quad \pa_{\delta}{F^m} \ll \pa_{\delta}{G^m},\quad m,s\in \{1,\ldots,n\}, \\
			&\mathbf{2}.\; \text{If } F \text{ and } G \text{ depend on a parameter } \delta \in [\delta_1,\delta_2], \text{ then}\\
			& \qquad \qquad \qquad \qquad\qquad\qquad\qquad \int_{\delta_1}^{\delta_2} F^m \, d \delta \ll \int_{\delta_1}^{\delta_2} G^m \, d \delta.
		\end{split}
	\end{equation*}
\end{lemma}

\begin{proof}
	This statement can be proved by a direct computation.
\end{proof}

\begin{lemma} \label{theineqforFkm}
	Let $\| F \|_{\rho} \leq c$. Then
	\begin{equation*}
		|F_{\bkk}^m| \leq c \rho^{-|\bkk|},\qquad \bkk \in \mZ^n_{\diamd},\qquad m \in \{1,\ldots,n\}.
	\end{equation*} 
\end{lemma}
\begin{proof}
	Take an arbitrary $\rho_0<\rho$. Using the Cauchy integral formula,
	\begin{equation*}
		F_{\bkk}^m = \frac{1}{(2\pi i)^n} \oint_{|z_1| = \rho_0} \,dz_1\ldots \oint_{|z_n|=\rho_0} \frac{F^m(\bzz)}{\bzz^{\bkk+\mathbf{1}}} \,d z_n,\qquad \mathbf{1}=(1,\ldots,1).
	\end{equation*}
	
	Therefore, we obtain
	\begin{equation*}
		|F_{\bkk}^m| \leq c \rho_0 ^{-|\bkk|} \text{ for any } \rho_0 < \rho.
	\end{equation*}
	
\end{proof}

\begin{lemma} \label{thefunctf}
	Let $F \in \calA^{\rho}$, $F = O_s(\bzz)$, $a = a(\rho)$, $F = \sum_{|\bkk| \geq s} F_{\bkk} \bzz^{\bkk}.$ Then
	\begin{equation*}
		F \ll \frac{a \rho \zeta^s}{\rho - \zeta},\qquad \zeta = \sum_{j=1}^n z_j.
	\end{equation*}
\end{lemma}
\begin{proof}
	Take $m \in \{1,\ldots,n\}$. Using Lemma \ref{theineqforFkm}, we have
	\begin{equation*}
		F^m = \sum_{|\bkk| \geq s} F_{\bkk}^m \bzz^{\bkk},\qquad |F_{\bkk}^m| \leq a \rho^{s-|\bkk|}, \qquad \bkk \in \mZ^n_{\diamd}.
	\end{equation*}
	
	Therefore,
	\begin{equation*}
		F^m \ll \sum_{j=s}^{\infty} \sum_{|\bkk|=j} a \rho^{s-|\bkk|} \bzz^{\bkk} = \sum_{j=s}^{\infty} a \rho^{s-j} \sum_{|\bkk|=j} \bzz^{\bkk}.
	\end{equation*}
	
	Since $\sum_{|\bkk|=j} \bzz^{\bkk} \ll \zeta^j,$ we obtain
	\begin{equation*}
		F^m \ll \sum_{j=s}^{\infty} a \rho^{s-j} \zeta^j = \frac{a \rho \zeta^s}{\rho - \zeta}.
	\end{equation*}
\end{proof}
\begin{lemma} \label{majornatlem}
	Suppose that $u \in \calA^{\rho}$, $\|u\|_{\rho} = c$, and $\{b_j\}_{j=1}^{\infty}$ is a sublinear sequence. Then for any $\alpha > - \ln \rho$ there exists $c>0$ such that
	\begin{equation}\label{estimforcoef}
		|U_{\bkk}^m| \leq c e^{b_{|\bkk|} + \alpha |\bkk|}.
	\end{equation}
\end{lemma}
\begin{proof}
	Using Lemma \ref{theineqforFkm}, we have $|F_{\bkk} ^m| \leq c_F \rho^{-|\bkk|}$. In \eqref{estimforcoef} choose $c$ so that
	\begin{equation*}
		c_F e^{-b_q - (\alpha + \ln \rho)q} < c,\qquad q \in \mZ_{+}.
	\end{equation*}
	
	For any $\alpha > -\ln \rho$ the function $q \mapsto  e^{-b_q - (\alpha + \ln \rho)q}$ is bounded, since $b_q$ is sublinear.
\end{proof}

\subsection{Majorant principle}

\begin{definition}
	An ordinary differential equation system
	\begin{equation*}
		\pa_{\delta} F^m = \Phi^m(F^1,\ldots,F^n), \qquad F^m = \sum_{|\bkk|\geq3} F_{\bkk}^m \bzz^{\bkk},\qquad m \in \{1,\ldots,n\} 
	\end{equation*}
	on $\calF$ has a nilpotent form if, for every $\bkk \in \mZ^n _{\diamd}$, we have $\pa_{\delta} F^m_{\bkk} = \Phi^m _{\bkk}$, where $\Phi^m _{\bkk}$ is a function depending on $F^1_{\mathbf{p}_1}, F^2_{\mathbf{p}_2},\ldots, F^n_{\mathbf{p}_n}$ with $|\mathbf{p}_s| < |\bkk|$ for every $s \in \{1,\ldots,n\}$.
\end{definition}

Consider the system of ordinary differential equations
\begin{equation}\label{diffeq1}
	\pa_{\delta} F^m = \Phi^m(F,\delta),\qquad F \big|_{\delta=0} = \widehat{F},\qquad F=(F^1,\ldots,F^n).
\end{equation}

Here $F \in \calF$ depends on the parameter $\delta$, and $\Phi$ is a mapping from $\calF \times \mR_{+}$ to $\calF$.

Associate with system (\ref{diffeq1}) the so-called majorant system
\begin{equation}\label{diffsystem2}
	\pa_{\delta} \bF^m = \Psi^m(\bF,\delta),\qquad \bF\big|_{\delta=0} = \widehat{\bF},\qquad \bF = (\bF^1,\ldots,\bF^n).
\end{equation}

Set $\Phi^m_{\bkk} = p^m_{\bkk} \circ \Phi$ and $\Psi^m _{\bkk} = p^m_{\bkk} \circ \Psi$.

\begin{definition} \label{defmaj}
	System (\ref{diffeq1}) is a majorant for (\ref{diffsystem2}) if the following two properties hold:
	\begin{equation*}
		\begin{split}
			&(a).\; \widehat{F} \ll \widehat{\bF},\\
			&(b).\ \text{For every } F \ll G \text{ and } \delta\geq0 \text{ we have } \Psi^{m}_{\bkk} (F,\delta) \ll \Phi^m _{\bkk} (\bF,\delta) \text{ for any } \bkk \in \mZ^{n}_{+}.
		\end{split}
	\end{equation*}
\end{definition}

\begin{theorem}[(Majorant principle)]
	Suppose there exists a solution $\bF = \bF(\cdot,\delta) \in \calA$ of system (\ref{diffsystem2}) on the interval $\delta\in [0,\delta_0]$. Then (\ref{diffeq1}) has a unique analytic solution $F$ on $[0,\delta_0]$, and moreover $F(\cdot,\delta) \ll \bF(\cdot,\delta)$.
\end{theorem}

\begin{theorem}
	Suppose systems (\ref{diffeq1}) and (\ref{diffsystem2}) have nilpotent form. Then the majorant principle holds.
\end{theorem}

\begin{proof}
	We prove this statement for system (\ref{reducedODE}). For system (\ref{reducedODE}) we have $|\bkk_0|=2$. The nilpotent form of (\ref{diffeq1}) implies
	\begin{equation*}
		0 = \pa_{\delta} F^m _{\bkk_0} \ll \pa_{\delta} \bF^m _{\bkk_0},\qquad m\in \{1,\ldots,n\}.
	\end{equation*}
	Hence $F_{\bkk_0}(\delta) \ll \bF_{\bkk_0}(\delta)$ for $\delta \geq 0$.
	
	We prove the theorem by induction on $|\bkk|$. Suppose $F_{\bkk}(\delta) \ll \bF_{\bkk}(\delta)$ for $\delta \geq 0$ when $|\bkk| < K$. For any $\bkk$ with $|\bkk|=K$, by the induction hypothesis and property (b) from Definition \ref{defmaj} we have
	\begin{equation*}
		\pa_{\delta} (\bF^m_{\bkk} - F^m_{\bkk}) = \Psi^m_{\bkk}(\bF(\cdot,\delta),\delta) - \Phi^m_{\bkk}(F(\cdot,\delta),\delta) \gg 0.
	\end{equation*}
	
	Therefore,
	\begin{equation*}
		\bF^m_{\bkk}(\delta) = \widehat{\bF}^m_{\bkk} (\delta) + \int_{0}^{\delta} \Psi^m_{\bkk} (\bF(\cdot,\lambda),\lambda) \, d\lambda \gg  F^m_{\bkk} (\delta) + \int_{0}^{\delta} \Phi^m_{\bkk} (F(\cdot,\lambda),\lambda) \, d\lambda.
	\end{equation*}
	Here we used that the arguments of $\Psi^m_{\bkk}$ and $\Phi^m_{\bkk}$ are known by the induction hypothesis.
\end{proof}

\subsection{Auxiliary estimates}

\begin{lemma} \label{axlemmas}
	Let the sequence $\{b_j\}_{j \in \mZ_{+}}$ be convex. Then
	
	(1) for any $1 \leq m < k < l$,
	\begin{equation}\label{convexseqineq1}
		(l-k)b_m + (m-l)b_k + (k-m) b_l \geq 0.
	\end{equation}
	
	(2) for any $1 \leq m < k \leq l$,
	\begin{equation}\label{convexseqineq2}
		b_k + b_l \leq b_{k-m} + b_{l+m}.
	\end{equation} 
\end{lemma}

\begin{proof}
	We prove the first assertion of the lemma. First, consider the case $l-k=1$. If $k-m=1$, then inequality (\ref{convexseqineq1}) coincides with the condition of convexity. We will prove the lemma by induction on $k-m$. Assume that (\ref{convexseqineq1}) holds for $k-m=q$. Then (\ref{convexseqineq1}) follows if we add the following inequalities
	\begin{equation*}
		\begin{split}
			(l-k)b_m &+ (m-l+1)b_{k-1} + (k-1-m)b_{l-1} \geq 0,\\
			&- (m-l+1)b_{k-1} + 2 (m-l+1)b_{l-1} - (m-l+1)b_l \geq 0,
		\end{split}
	\end{equation*}
	which are valid by the induction hypothesis. The case $l-k>1$ follows from the case $l-k=1$ by an analogous induction.
	
	We prove the second inequality. From (\ref{convexseqineq1}) it follows that
	\begin{equation*}
		(l-k) b_{k-m} + (k-m-l)b_k + m b_l \geq 0,
	\end{equation*}
	\begin{equation*}
		m b_k + (k-l-m)b_l + (l-k)b_{l+m} \geq 0,
	\end{equation*}
	from which (\ref{convexseqineq2}) follows.
\end{proof}

\begin{lemma} \label{aboutexistaj}
	
	1. For any convex sublinear sequence $\{b_s\}$, the sequence $\{a_j\}$ satisfying (\ref{seqaj}) is a Bruno sequence.
	
	2. For any Bruno sequence $\{a_j\}$ satisfying (\ref{seqaj}), there exists a sublinear, convex, negative, nonincreasing sequence $\{b_j\}$.
\end{lemma}

\begin{proof}
	We prove the first assertion. Note that
	\begin{equation*}
		\ln a_j - \ln a_{j-1} = b_{r_{m+1}} - 3 b_{r_m} + 2 b_{r_{m-1}},\qquad r_m = 2^{m-1} + 1. 
	\end{equation*}
	
	Using the first assertion of Lemma \ref{axlemmas}, we obtain
	\begin{equation*}
		\ln a_j - \ln a_{j-1} \geq 0,
	\end{equation*}
	so the sequence $\{a_j\}$ is nondecreasing. Since (\ref{seqaj}) holds, we have
	\begin{equation*}
		\sum_{j=1}^{J} 2^{-j} \ln a_j = -2 b_{r_1} + 2^{-J} b_{r_{J+1}}.
	\end{equation*}
	By sublinearity of the sequence, the series $\sum_{j=1}^{+\infty} 2^{-j} \ln a_j$ converges.
	
	We prove the second assertion. Set $A = 2^n n \Omega_{2r-1}$. First, we prove negativity of the sequence. Define
	\begin{equation*}\label{simpequation1}
		b_{r_{j+1}} = 2^j b_{r_1} + \sum_{s=1}^{j} 2^{j-s} \ln\big(  2^{-s} r^n_{j-s} A\big). 
	\end{equation*}
	
	The last equation is equivalent to
	\begin{equation*}
		\frac{b_{r_{j+1}}}{2^j} = b_{r_1} + \sum_{s=1}^{j} 2^{-s} \ln  \big(2^{-s} r_{j-s}^n A\big). 
	\end{equation*}
	
	Choosing $b_{r_1} = \sum_{s=1}^{+\infty} 2^{-s} \ln(  2^{-s} r_{j-s}^n A)$, we get
	\begin{equation}\label{provingsublinearity}
		b_{2^j +1} = - 2^j\sum_{s=j+1}^{+\infty} 2^{-s} \ln\big(  2^{-s} (2^{j-s-1} + 1)^n A\big) < 0.
	\end{equation}
	
	Using (\ref{provingsublinearity}), the sequence $\{b_{2^j + 1}\}_{j=1}^{\infty}$ is sublinear. To prove that the sequence is nonincreasing, consider
	\begin{equation*}\begin{split}
			b_{r_{j}} - b_{r_{j-1}} &= - 2^{j+1}\sum_{s=j+2}^{+\infty} 2^{-s} \ln\big(  2^{-s} r_{j-s}^n A\big) +  2^j\sum_{s=j+1}^{+\infty} 2^{-s} \ln\big(  2^{-s} r_{j-s}^n A\big) \\
			&= -2^j \sum_{s=j+2}^{+\infty} 2^{-s}\ln\big(  2^{-s} r_{j-s}^n A\big) + 2^{-1} \ln\big(  2^{j+1} r_{j+1}^n A\big)\\
			&\leq -2^{-1} \ln\big(  2^{-s} r_{j-s}^n A\big) + 2^{-1} \ln\big(  2^{-s} r_{j-s}^n A\big) =0. 
		\end{split} 
	\end{equation*}
	
	To prove convexity, verify the inequality
	\begin{equation*}
		b_{r_{j+1}} \leq \frac{b_{r_{j+2}} - b_{r_{j}}}{2^j + 2^{j-1}} (2^j - 2^{j-1}) + b_{r_j},
	\end{equation*}
	which is equivalent to
	\begin{equation*}
		2^{j-1} (b_{r_{j+1}} - 2 b_{r_{j}}) -2^{j-1} (b_{r_{j+2}} - 2 b_{r_{j+1}}) \leq 0.
	\end{equation*}
	
	Using (\ref{seqaj}) and the Bruno sequence conditions, we obtain that the latter is equivalent to
	\begin{equation*}
		2^{j-1} (a_j - a_{j+1}) \leq 0.
	\end{equation*}
\end{proof}

\end{document}